\documentclass[11pt,letterpaper]{amsart}
\usepackage{amsmath,amssymb,amsthm,amscd}
\usepackage{graphicx}
\usepackage{tikz}
\usepackage{color}
\numberwithin{equation}{section}
\usepackage[plainpages=false]{hyperref}
\newtheorem{prop}{Proposition}[section]
\newtheorem{thm}[prop]{Theorem}
\newtheorem{lem}[prop]{Lemma}
\newtheorem{coro}[prop]{Corollary}
\newtheorem{rem}[prop]{Remark}
\newtheorem{defi}[prop]{Definition}

\def\begeq{\begin{equation}}
\def\endeq{\end{equation}}

\begin{document}
\title[Equivariant $\mathbb R$-test configurations of polarized spherical varieties]
{Equivariant $\mathbb R$-test configurations of polarized spherical varieties}
\author[Yan Li and ZhenYe Li]{Yan Li$^{*1}$ and ZhenYe Li $^{*2}$}

\address{$^{*1}$School of Mathematics and Statistics, Beijing Institute of Technology, Beijing, 100081, China.}
\address{$^{*2}$College of Mathematics and Physics, Beijing University of Chemical Technology, Beijing, 100029, China.}
\email{liyan.kitai@yandex.ru,\ \ \ lizhenye@pku.edu.cn}

\thanks {$^{*1}$Partially supported by NSFC Grant 12101043 and the Beijing Institute of Technology Research Fund Program for Young Scholars.}
\thanks {$^{*2}$Partially supported by NSFC Grant 12001032.}



\begin{abstract}
Let $G$ be a connected, complex reductive Lie group and $G/H$ a spherical homogeneous space. Let $(X,L)$ be a polarized $G$-variety which is
a spherical embedding of $G/H$. In this paper we classify $G$-equivariant normal $\mathbb R$-test configurations of $(X,L)$ via combinatorial data. In
particular we classify the special ones, and prove a finiteness theorem of central fibres of $G$-equivariant special $\mathbb R$-test configurations.
Also, as an application we study the semistable degeneration problem of a $\mathbb Q$-Fano spherical variety.
\end{abstract}
\maketitle

\section{Introduction}
The notation of K-stability was first introduced by \cite{Tian-97} in terms of degenerations and re-formulated in \cite{Do} in terms of ($\mathbb Z$-)test configuration. The famous Yau-Tian-Donaldson conjecture confirms that the existence of K\"ahler-Einstein metrics on a Fano manifold is equivalent to the K-stability \cite{Tian-15,CDS}.
There are also variants of K-stability and Yau-Tian-Donaldson conjecture for other canonical metrics (see \cite{Do, WZZ, Han-Li-KRS}, etc). All these notations are phrased by sign of certain invariant of ($\mathbb Z$-)test configurations. Roughly speaking, a ($\mathbb Z$-)test configuration of a polarized variety $(X,L)$ is a flat family
$(\mathcal X,\mathcal L)$ over $\mathbb C$ with a $\mathbb C^*$-equivariant projection $\pi:\mathcal X\to\mathbb C$ so that the fibre $(\mathcal X_t,\mathcal L_{\mathcal X_t})\cong(X,L^{r_0})$ for a fixed exponent $r_0\in\mathbb N_+$ at any $t\not=0$. The fibre $\mathcal X_0$ at $0$ is called the
\emph{central fibre}. A test configuration is further called \emph{special} if its central fibre is normal. It is proved by \cite{Li-Xu-2014} that to test its K-stability of a $\mathbb Q$-Fano variety, it suffices to check only special test configurations. Concerning other canonical metrics on a $\mathbb Q$-Fano variety, such as K\"ahler-Ricci soliton, or more general, weighted soliton metric, \cite{WZZ, Han-Li-KRS} generalized this result to various of modified or weighted equivariant K-stability with respective to a reductive group action.

The theorems cited above lead to a beautiful result on toric Fano manifolds. It is known from \cite{Do} that any equivariant special test configuration of a toric manifold $X$ is a product test configuration. That is, the total space $\mathcal X$ is $\mathbb C^*$-equivariantly isomorphic to $X\times\mathbb C$. In fact, such a test configuration is induced from a toric vector field on $X$. In particular the central fibre $\mathcal X_0$ is always isomorphic to $X$. Thus a toric Fano manifold is equivariantly K-stable (hence admits K\"ahler-Einstein metrics by \cite{Tian-15, CDS}) if and only if its Futaki invariant vanishes. We remark that this result was first proved by \cite{Wang-Zhu} via a PDE approach. Also \cite{WZZ} proved that a toric Fano manifold is always equivariantly modified K-stable, hence it always admits a K\"ahler-Ricci soliton. In fact, these results are true for a more general class of horospherical Fano manifolds (cf. \cite[Theorem 5.3]{Del3}). In all these cases, to check the stability, it is crucial to use the fact that the central fibre of an equivariant special test configuration is always isomorphic to the original manifold.

On the other hand, the ($\mathbb Z$-)test configurations of a polarized variety can also be formulated in the language of filtrations on its Kodaira ring (cf. \cite{Witt-Ny}). This leads to a generalized notation of $\mathbb R$-test configurations (cf. \cite{Popov-1986, Han-Li, Blum-Liu-Xu-Zhuang} and references there in) which play an important role in proving algebraic uniqueness of the limit of (normalized) K\"ahler-Ricci flow on a $\mathbb Q$-Fano variety \cite{Blum-Liu-Xu-Zhuang, Chen-Wang-Sun, Han-Li}. In this paper, we work with the spherical varieties, which include toric varieties, horospherical varieties, etc. We will consider the equivariant normal $\mathbb R$-test configurations of general polarized spherical varieties, which need not to be $\mathbb Q$-Fano. Let $G$ be a connected, complex reductive group and $(X,L)$ be a polarized $G$-spherical variety. Denote by $\mathcal V$ its valuation cone (for precise meaning of the terminologies, see Section 2 below). We have the classification of equivariant special $\mathbb R$-test configurations,
\begin{thm}\label{main-thm-3}
The $G$-equivariant $\mathbb R$-test configurations of $(X,L)$ are in one-one correspondence with (real) vectors in the cone $\mathcal V_\mathbb R:=\mathcal V\otimes_\mathbb Z\mathbb R$.
\end{thm}


For a $G$-equivariant special $\mathbb R$-test configuration of $(X,L)$, we further compute the combinatorial data of its central fibre $\mathcal X_0$, which includes the spherical datum of the spherical homogeneous space $G/H_0$ embedded in $\mathcal X_0$ and the its coloured fan. According to a deep uniqueness result of Losev \cite{Losev} (see also \cite[Theorem 30.22]{Timashev-book}) and the Luna-Vust theory \cite{Luna-Vust} (see also \cite{Timashev-book}), the above data completely determines $\mathcal X_0$. Using an approximating result (see Proposition \ref{F-Lambda-appro} below) and a computation method of \cite{Ga-Ho-datum}, we get the following finiteness result;

\begin{thm}\label{main-thm-5}
Let $(X,L)$ be a polarized $G$-spherical variety. Up to $G$-equivariant isomorphism, the set of central fibres of all $G$-equivariant special $\mathbb R$-test configurations of $(X,L)$ is finite. More precisely, the possible central fibres are in one-one correspondence with the faces of $\mathcal V_\mathbb R$.
\end{thm}

A full statement of Theorem \ref{main-thm-5} is Theorem \ref{finiteness-X0} below. Note that for a horospherical variety (in particular, a toric variety) $X$, its valuation cone $\mathcal V$ is an affine space. Hence $\mathcal V_\mathbb R$ has itself as its only face. On the other hand, the origin $O\in\mathcal V$ defines a product (even trivial) test configuration with central fibre isomorphic to $X$. By Theorems \ref{main-thm-3}-\ref{main-thm-5}, the central fibre of any equivariant special $\mathbb R$-test configuration is isomorphic to $X$, as showed in \cite{Do,Del3}.

The proof of the above theorems actually depends on the classification of general equivariant $\mathbb R$-test configurations of polarized spherical varieties. In general the collection of even $\mathbb Z$-test configurations is extremely huge. However, it can stilled be expected that there is a simpler classification of equivariant test configurations on varieties with large symmetry. The classification of equivariant ($\mathbb Z$-)test configurations of polarized toric varieties was established by \cite{Do}.
Later \cite{AB2,AK} prove the classification theorem for polarized stable reductive varieties and then \cite{Del-2020-09} for polarized spherical varieties.
One key observation is that by adding a trivial fibre at infinity so that one compactifies the ($\mathbb Z$-)test configuration to a family $(\bar{\mathcal X},\bar{\mathcal L})$ over $\mathbb{CP}^1$, the total space $(\bar{\mathcal X},\bar{\mathcal L})$ is also a polarized toric (stable reductive, spherical, respectively) variety, provided $(\mathcal X,\mathcal L)$ is ample. In this way \cite{Do} characterized equivariant ($\mathbb Z$-)test configurations of a polarized toric variety $(X,L)$ via concave rational piecewise linear functions on the moment polytope of $(X,L)$. The result of stable reductive varieties or general spherical varieties is similar except a restriction that
the gradient of this piecewise linear function should lie in the valuation cone of the variety (cf. \cite{AB2,AK,Del-2020-09}).

In this paper, we will first classify the equivariant $\mathbb R$-test configurations of a polarized spherical variety. We used a totally algebraic approach since the total space of an $\mathbb R$-test configuration is usually very large and complicated.
Our way is to directly study the corresponding Rees algebra which contains the full information of the total space. We have the following classification,
\begin{thm}\label{main-thm-1}
Let $(X,L)$ be a polarized $G$-spherical variety with $\Delta_L$ its moment polytope. The $G$-equivariant $\mathbb R$-test configurations of $(X,L)$ are in one-one correspondence with the pairs of form $(f,\Gamma)$, where
\begin{itemize}
\item $f$ is a concave piecewise linear function on $\Delta_L$ so that $\nabla f\in\mathcal V$, and domains of linearity of $f$ consist of rational polytopes in $\Delta_L$;
\item $\Gamma$ is the group generated by the points of discontinuity of the $\mathbb R$-test configuration on $R.$
\end{itemize}
\end{thm}
The precise statement of Theorem \ref{main-thm-1} is Theorem \ref{G-classify} below. The Theorem was first proved for polarized group compactifications in \cite{LL-arXiv-2021}.
In Section 3.2 we generalize the argument from \cite{LL-arXiv-2021} to show Theorem \ref{G-classify}.
One crucial step is to compute the tail vectors appeared in multiplication of isotypic components of the Kodaira ring $R$ (see Proposition \ref{H0(G/H-Lk)}).
This is a generalization of the multiplication rule of the coordinate ring of a reductive Lie group \cite[Proposition 4]{Timashev-Sbo}. We used a technique from \cite{Timashev-survey, Timashev-book} to overcome the
difficult when a spherical homogeneous space is not quasiaffine.

For equivariant $\mathbb R$-test configurations with reduced central fibre, the classification is much simpler,
\begin{thm}\label{main-thm-2}
The $G$-equivariant $\mathbb R$-test configurations of $(X,L)$ are in one-one correspondence with the concave piecewise linear functions on $\Delta_L$ so that its gradient lies in $\mathcal V$, and its domains of linearity consist of rational polytopes in $\Delta_L$.
\end{thm}
Indeed, the points of discontinuity are totally determined by $f$ in this case. The precise statement of Theorem \ref{main-thm-2} is Theorem \ref{G-classify-reduced}. Then we conclude Theorem \ref{main-thm-3} from Theorem \ref{main-thm-2}, see Proposition \ref{R-TC-Lambda} below for details.

Back to the $\mathbb Q$-Fano cases, our results apply to the semistable degeneration problem, and we will prove
\begin{thm}\label{main-thm-4}
Let $X$ be a $\mathbb Q$-Fano spherical variety which is a completion of a spherical homogenous space $G/H$. Denote by $\Delta_+$ the moment polytope of $(X,-K_X)$. Then there is a unique $G$-equivariant special $\mathbb R$-test configuration $\mathcal F_{\Lambda_0}$
corresponding to $\Lambda_0\in\mathcal V_\mathbb R(G/H)$ which is the semistable degeneration of $X$. Moreover, if the weighted barycenter of $\Delta_+$,
$$\mathbf{b}(\Lambda_0):=\frac{\int_{\Delta_+}y_ie^{-\Lambda_0(y)}\pi dy}{\int_{\Delta_+}e^{-\Lambda_0(y)}\pi dy}\in\kappa_P+{\rm RelInt}{(-V_\mathbb R(G/H_0)^\vee)}.$$
Then the central fibre $\mathcal X_0$ is the limiting space of the normalized K\"ahler-Ricci flow on $X$. Here $\kappa_P$ is the $B$-weight of the canonical $B$-semiinvariant section of $-K_X$, and $\pi$ is a weight function on $\Delta_+$ determined by $X$.
\end{thm}
Theorem \ref{main-thm-4} is a combination of Theorem \ref{H-inv-F-thm} and Proposition \ref{limit-KRF} below. For more details on this topic for spherical varieties, we refer to the readers \cite{Li-Wang}.
Also, we apply Theorem \ref{main-thm-1} to study the base change progress in Section 6.2, for which so far we did not find an explicit statement in the literatures (see Proposition \ref{base-change} below).

Combining with the arguments of Section 4, we can conclude the following result of the K\"ahler-Ricci limiting problem,
\begin{thm}\label{main-thm-6}
Let $X$ be a $\mathbb Q$-Fano spherical variety which is a completion of a spherical homogenous space $G/H$. Assume that $X$ does not admit any K\"ahler-Ricci soliton. Then the limiting space of the normalized K\"ahler-Ricci flow on $X$ can not be a completion of $G/H$.
\end{thm}

The paper is organized as following: In Section 2 we review preliminaries on $\mathbb R$-test configurations and spherical varieties; Section 3 is devoted to the classification Theorems \ref{main-thm-3}, \ref{main-thm-1} and \ref{main-thm-2}.
In Section 4 we compute the combinatory data of the central fibre of a $G$-equivariant special $\mathbb R$-test configuration and prove Theorem \ref{main-thm-5}.
In the Section 5 we give applications of the main Theorems, especially we apply our results to the semistable degeneration problem and prove Theorems \ref{main-thm-4}-\ref{main-thm-6}. In the Appendix we include a useful lemma.

\subsection*{Acknowledgement} The authors would like to sincerely thank Prof. Gang Tian for comments and suggestions to this work.

\section{Preliminaries}

\subsection{Filtrations and $\mathbb R$-test configurations}
In this section we recall some basic notations of filtrations and $\mathbb R$-test configurations. Up to the authors' knowledge, filtrations of homogeneous coordinate ring (Kodaira ring) of a projective $G$-varieties were first investigated in \cite{Popov-1986}, where a canonical horospherical degeneration of a $G$-variety was constructed.
We refer to the readers \cite{Popov-1986, Witt-Ny} and \cite[Section 2]{Han-Li} for further knowledge.

Let $X$ be a projective variety and $L$ a very ample line bundle over $X$ so that $|L|$ gives a Kodaira embedding of $M$ into some projective space $\mathbb P^N$. The Kodaira ring of $M$ is
$$R(X,L) =\oplus_{k\in\mathbb N}R_k,~\text{where}~R_k={\rm H}^0(X,L^k).$$
\begin{defi}\label{filtrantion-def}
A filtration $\mathcal F$ of $R$ is a family of subspaces $\{\mathcal F^sR_k\}_{s\in\mathbb R,k\in\mathbb N}$ of $R(M,L)=\oplus_{k\in\mathbb N}R_k$ such that
\begin{itemize}
\item[(1)] $\mathcal F$ is decreasing: $\mathcal F^{s_1}R_k\subset \mathcal F^{s_2}R_k,~\forall s_1\geq s_2\text{ and }k\in\mathbb N;$
\item[(2)] $\mathcal F$ is left-continuous: $\mathcal F^sR_k=\cap_{t<s}\mathcal F^tR_k,~\forall k\in\mathbb N;$
\item[(3)] $\mathcal F$ is linearly bounded: There are constants $c_\pm\in\mathbb Z$ such that for each $k\in\mathbb N$, such that $$\mathcal F^sR_k=0,~\forall s>c_+k\text{ and }\mathcal F^sR_k=R_k,~\forall s<c_-k;$$
\item[(4)] $\mathcal F$ is multiplicative: $\mathcal F^{s_1}R_{k_1}\cdot\mathcal F^{s_2}R_{k_2}\subset\mathcal F^{s_1+s_2}R_{k_1+k_2},$ for all $k_1,k_2\in\mathbb N$ and $s_1,s_2\in\mathbb R.$
\end{itemize}
\end{defi}
Let $\Gamma(\mathcal F,k)$ be the set of values of $s$ where the filtration of $R_k$ is discontinuous, and $\Gamma(\mathcal F)$ the Abelian group generated by the union of all $\Gamma(\mathcal F,k)$'s. Also, set
\begin{align}\label{group-discon}
\Gamma_+(\mathcal F):=\cup_k(\Gamma(\mathcal F,k)-\min\Gamma(\mathcal F,k)),
\end{align}
and $\hat\Gamma(\mathcal F)$ the Abelian group generated by $\Gamma_+(\mathcal F)$. Then $\hat\Gamma(\mathcal F)\leq\Gamma(\mathcal F)$. We associate to each filtration $\mathcal F$ two graded algebra,
\begin{defi}\label{def-rees-gr}
\begin{itemize}
\item[(1)]
The Rees algebra,
\begin{align}\label{rees-def}
{\rm R}(\mathcal F):=\oplus_{k\in\mathbb N}\oplus_{s\in\Gamma(\mathcal F)}t^{-s}\mathcal F^sR_k,
\end{align}
and
\item[(2)]
The associated graded ring of $\mathcal F$,
\begin{align}\label{GrF-def}{\rm Gr}
(\mathcal F):=\oplus_{k\in\mathbb N}\oplus_{s\in\Gamma(\mathcal F)}(\mathcal F^sR_k/\mathcal F^{>s}R_k).
\end{align}
\end{itemize}
\end{defi}

There is an important class of filtrations, called the $\mathbb R$-test configurations, which can be considered as a generalization of the usual ($\mathbb Z$-)test configurations introduced in \cite{Do}.

\begin{defi}
Set $\Gamma_{\leq0}(\mathcal F):=\Gamma(\mathcal F)\cap\mathbb R_{\leq0}$. When ${\rm R}(\mathcal F)$ is finitely generated over $\oplus_{s\in\Gamma_{\leq0}(\mathcal F)}t^{-s}R_0$, we say that $\mathcal F$ is an $\mathbb R$-test configuration of $(X,L)$. In this case, ${\rm Gr}(\mathcal F)$ is also finitely generated. The projective scheme
$$\mathcal X_0:={\rm Proj}({\rm Gr}(\mathcal F))$$
is called the central fibre of $\mathcal F$.
\end{defi}
Set $\mathcal B:={\rm Spec}(\mathbb C[\Gamma(\mathcal F)\cap\mathbb R_{\geq0}])$. According to \cite[Section 2]{Teissier}, the total space $\mathcal X$ of $\mathcal F$ is the projective scheme
\begin{align}\label{total-space-ring}
\mathcal X={\rm Proj}_{\mathcal B}({\rm R}(\mathcal F)),
\end{align}
and ${\rm Gr}(\mathcal F)$ is the quotient of ${\rm R}(\mathcal F)$ by the ideal generated by $\Gamma_{<0}(\mathcal F):=\Gamma(\mathcal F)\cap\mathbb R_{<0}$. When $\mathcal F$ is an $\mathbb R$-test configuration, the Abelian group $\hat\Gamma(\mathcal F)$ generated by \eqref{group-discon} also has finite rank. Denote its rank by $\hat r_\mathcal F$.
The $\Gamma(\mathcal F)$-grading of ${\rm Gr}(\mathcal F)$ corresponds to a (possibly real) holomorphic vector field $\Lambda$ on $\mathcal X_0$, which generates a rank $\hat r_\mathcal F$ torus (denote by $\mathbb T$) action. In fact, this $\mathbb T$-action is the restriction of a rank $\hat r_\mathcal F$-torus (still denoted by $\mathbb T$) action from ${\rm PSL}_N(\mathbb C)$, the group of holomorphic actions on $\mathbb P^N$. Moreover, the total space $\mathcal X$ is a completion of $\mathbb T\cdot X$ in $\mathbb P^N$. The relation \eqref{total-space-ring} gives the embedding of $\mathcal X$ via the restriction of $\mathcal O_{\mathbb P^N}(1)$ on $\mathcal X$. Note that we take the convention that the $\exp{(t\Lambda)}$-action has weight $t^{-s}$ on the $(\mathcal F^sR_k/\mathcal F^{>s}R_k)$-piece in \eqref{GrF-def}.

\begin{rem}\label{F-normalized}
By finite generation, $\mathcal F$ (as a filtration on $R(X,L^{k_0})$) is generated by $\mathcal FR_{k_0}$ for some $k_0\in\mathbb N_+$. By a shifting of $\mathcal F$, we can always normalize $\min\Gamma(\mathcal F,k_0)=0$. Then $\Gamma (\mathcal F)$ contains all points of discontinuity of $\mathcal F$. In the following, all $\mathbb R$-test configurations are assumed to be normalized in this way.
\end{rem}

\begin{rem}
When ${\rm rank}(\Gamma(\mathcal F))=1$, up to a rescaling we can embed $\Gamma(\mathcal F)$ in $\mathbb Z$. The above $\mathbb R$-test configuration is simply the usual ($\mathbb Z$-)test configuration introduced in \cite{Do}. For detailed correspondence, see \cite{Witt-Ny} and \cite[Section 2.2]{Han-Li}.
\end{rem}

There is an important subclass of normal $\mathbb R$-test configurations,
\begin{defi}
An $\mathbb R$-test configuration $\mathcal F$ is called special if the (scheme-theoretic) central fibre $\mathcal X_0$ is a normal variety.
\end{defi}

\subsubsection{Equivariant $\mathbb R$-test configurations}
Let $(M,L)$ be a polarized variety with a group $\mathfrak G$-action. Then $\mathfrak G$-acts on its Kodaira ring. Let $\mathcal F$ be a filtration on $R$. Define the action of $\mathfrak G$ on $\mathcal F$ by
\begin{align}\label{group-act-F}
(\sigma\cdot \mathcal F)^sR_k:=\sigma(\mathcal F^sR_k),~\forall s\in\mathbb R,k\in\mathbb N,
\end{align}
for any $\sigma\in \mathfrak G$. Clearly $\sigma\cdot \mathcal F$ is also a filtration on $R$, and it an $\mathbb R$-test configuration if and only if $\mathcal F$ is. As a generalization of equivariant $\mathbb Z$-test configurations, we define

\begin{defi}
A filtration $\mathcal F$ is called {$\mathfrak G$-equivariant} if $\mathcal F^sR_k$ is a $\mathfrak G$-invariant space of $R_k$ for any $s\in\mathbb R$ and $k\in\mathbb N$. That is
\begin{align}\label{group-eq-F}
\sigma(\mathcal F^sR_k)=\mathcal F^sR_k,~\forall s\in\mathbb R,k\in\mathbb N.
\end{align}
\end{defi}

As in the case of equivariant $\mathbb Z$-test configurations, it is clear that when $\mathcal F$ is a $\mathfrak G$-equivariant $\mathbb R$-test configuration, the projection from $\mathcal X$ to the base $\mathcal B$ is $\mathfrak G$-equivariant. Note that here $\mathfrak G$ acts on $\mathcal B$ trivially.

\subsubsection{Filtrations and semi-valuations}
Let $\mathcal F$ be a filtration on $R$. For any $\sigma_k\in R_k,~k\in\mathbb N_+$, set
\begin{align*}
 {\mathfrak v}_\mathcal F(\sigma_k):=\max\{s|\sigma_k\in\mathcal F^sR_k\},
\end{align*}
and for any $\sigma:=\sum_{k\in\mathbb N_+}\sigma_k$ with $(0\not=)\sigma_k\in R_k$, set
\begin{align*}
{\mathfrak v}_\mathcal F(\sigma):=\min\{\mathfrak v_\mathcal F(\sigma_k)|0\not=\sigma_k\in R_k\}.
\end{align*}
Then ${\mathfrak v}_\mathcal F(\cdot)$ defines a semi-valuation on $R$ (cf. \cite[Section 2.2]{Han-Li}). Conversely, given any valuation ${\mathfrak v}$ with finite log-discrepancy, there is a filtration $\mathcal F_{({\mathfrak v})}$ on $R$ by (cf. \cite[Example 2.2]{Han-Li}) satisfying the above two relations,
$$\mathcal F_{({\mathfrak v})}^sR_k:=\{\sigma\in R_k|{\mathfrak v}(\sigma)\geq s\}.$$

It is known that when ${\rm Gr}(\mathcal F)$ is an integral domain, ${\mathfrak v}_\mathcal F$ is a valuation and $\mathcal F=\mathcal F_{( {\mathfrak v}_\mathcal F)}$, up to a shifting. In particular, this applies to special $\mathbb R$-test configurations (cf. \cite[Lemma 2.11]{Han-Li}).

\subsection{Approximation of an $\mathbb R$-test configuration}
In the following, we will prove a general approximating result of $\mathbb R$-test configurations, which will be used later in the study of central fibre.

\begin{prop}\label{RTC-appr}
Let $(X,L)$ be a polarized projective variety and $\mathcal F$ an $\mathbb R$-test configuration of $(X,L)$. Then there is a $\mathbb Z$-test configurations $\mathcal F'$ of $(X,L)$ which has the same central fibre with $\mathcal F$. Moreover, if $(X,L)$ admits a reductive group $\mathfrak G$-action and $\mathcal F$ is $\mathbf G$-equivariant, then we can further choose $\mathcal F'$ to be $\mathbf G$-equivariant.
\end{prop}

\begin{proof}
Suppose that $L$ is very ample and $M$ is embedded in $\mathbb {CP}^N=\mathbb P({\rm H}^0(X,L))$. Then $R(X,L)$ is generated by $R_1$. Suppose that $\mathcal F$ is an $\mathbb R$-test configuration of $(X,L)$. Then geometrically, there is a (real) holomorphic vector field $\mathbf{\Lambda}$ which generates a rank $r_\mathcal F$-torus $\mathbf{T}\subset{\rm PSL}_{N}(\mathbb C)$ so that $\mathcal F$ is generated by the following filtration on $R_{1}:={\rm H}^0(X,L)$,
\begin{align}\label{T-decomp-R1}
\mathcal F^sR_{1}=\oplus_{\eta\in{\rm Lie}^*(\mathbf{T}),\mathbf{\Lambda}(\eta)\geq s}R_{1,\eta},
\end{align}
where $\mathbf{T}$ acts on the $R_{1,\eta}$-piece of the linear space $R_{1}$ with weight $\eta\in{\rm Lie}^*(\mathbf{T})$ (cf. \cite[Geometric $\mathbb R$-TC I]{Han-Li}). 

According to \cite[Section 3.1]{Chen-Wang-Sun}, there is a surjective map
$$\mathbf{\Phi}: \oplus_{k\in\mathbb N}{\rm Sym}^kR_1\to{\rm Gr}(\mathcal F),$$
whose kernel is the initial ideal $\bar I$ of the saturated ideal $I$ defining $X$. That is,  the ideal
generated by the initial terms of elements in $I$ with respect to the $s$-grading induced by $\mathcal F$. Then $\bar I$ is finitely generated and a set of generators is contained in $\oplus_{0\leq j\leq k_0}{\rm Sym}^jR_1$.

Consider the $\mathbf T$-action on $R_1$. Choose a basis $\{e_1,...,e_{N+1}\}$ of $R_1$ that is compatible with the weight decomposition \eqref{T-decomp-R1}. Then each $e_i\in R_{1,\eta_i}$, has $s$-grading $s_i=\mathbf{\Lambda}(\eta)$, and any $l$-product $\prod_{j=1}^le_{i_j}$ has $s$-grading $s_i=\sum_{j=1}^l\mathbf{\Lambda}(\eta_{i_j})$. Hence $\mathcal F$ induces an $s$-grading on $\oplus_{k\in\mathbb N}{\rm Sym}^kR_1$. And $\bar I$ is the initial ideal of $I$ under this grading. 
We claim that there is a cone $\mathfrak Z$ in ${\rm Lie}(\mathbf T)$ so that for any $\mathbf \Lambda'\in\mathfrak Z$, the $\mathbb R$-test configuration $\mathcal F'$ generated by
$$\mathcal F'^sR_{1}=\oplus_{\eta\in{\rm Lie}^*(\mathbf{T}),\mathbf{\Lambda'}(\eta)\geq s}R_{1,\eta}$$
induces a initial ideal $\bar I'=\bar I$, and consequently $\mathcal F'$ has the same central fibre with $\mathcal F$. Indeed, it suffices to solve the following inequality for $\mathbf\Lambda'\in{\rm Lie}(\mathbf T)$ so that the $s$-grading of $\mathcal F'$ on each ${\rm Sym}^jR_1,~j=1,...,k_0$ has the same order with that of $\mathcal F$. Denote
$$\mathfrak Y_1:=\{\eta\in{\rm Lie}^*(\mathbf T)|R_{1,\eta}\not=\{O\}\},$$
and for each $k\in\mathbb N_+$,
$$\mathfrak Y_k:=\{\sum_{i=1}^k\eta_i|\eta_i\in\mathfrak Y_1\}.$$
Then it suffices to require,
\begin{align}
\mathbf \Lambda'(\eta_{(j)1}-\eta_{(j)2})=&0,~\forall\eta_{(j)1},\eta_{(j)2} \in\mathfrak Y_j~\text{so that}~\mathbf \Lambda(\eta_{(j)1}-\eta_{(j)2})=0,\label{eq-Lambda-1}\\
\mathbf \Lambda'(\eta_{(j)1}-\eta_{(j)2})>&0,~\forall\eta_{(j)1},\eta_{(j)2} \in\mathfrak Y_j~\text{so that}~\mathbf \Lambda(\eta_{(j)1}-\eta_{(j)2})>0,\label{eq-Lambda-2}
\end{align}
for $j=1,...,k_0$. Note that the coefficients $\eta_{(j)1},\eta_{(j)2}$ in each (in-)equality are rational.
Hence the system \eqref{eq-Lambda-1}-\eqref{eq-Lambda-2} defines a rational convex cone $\mathfrak Z$ in ${\rm Lie}(\mathbf T)$ containing $\mathbf \Lambda$.
In fact, $\mathbf \Lambda$ is a generic element in $\mathfrak Z$.

Choose a rational element $\Lambda'\in\mathfrak Z$ which generates a one dimensional torus $(\mathbb C^*\cong)\mathbf T'\subset\mathbf T$. Consider the $\mathbb Z$-test configuration $\mathcal F'$ induced by $\mathbf T'$ and the corresponding initial ideal $\bar I'$ of $I$ induced by $I$. Clearly, $\bar I\cap R_j=\bar I'\cap R_j$ for $j=1,...,k_0$. Hence the generators of $\bar I$ is contained in $\bar I'$ and $\bar I\subset\bar I'$. On the other hand, for every $k\in\mathbb N$, the $k$-th pieces in ${\rm Gr}(\mathcal F)$ and ${\rm Gr}(\mathcal F')$ are isomorphic as linear spaces (regardless the ring structure). Hence $\bar I\cong\bar I'$ and the two $\mathbb R$-test configurations have the isomorphic central fibre.


Finally if $(X,L)$ admits a reductive group $\mathbf G$-action and $\mathcal F$ is $\mathbf G$-equivariant, we can lift the $\mathbf G$-action on $\mathbb{CP}^N$. That is, there is a morphism from $\mathbf G$ to ${\rm PSL}_N(\mathbb C)$ so that its restriction on $X$ gives the original $\mathbf G$-action on $X$. Since $\mathcal F$ is $\mathbf G$-equivariant, $\mathbf T$ commutes with $\mathbf G$. By our construction, each $\mathcal F'$ is also $\mathbf G$-equivariant.

\end{proof}

\subsection{Spherical varieties}
In the following we overview the theory of spherical varieties. The origin goes back to \cite{Luna-Vust}. We use \cite{Timashev-survey, Timashev-book} as main references.
\begin{defi}
Let $G$ be a connected, complex reductive group. A closed subgroup $H\subset G$ is called a {spherical subgroup of $G$} if there is a Borel subgroup $B$ of $G$ acts on $G/ H$ with an open orbit.
In this case $G/H$ is called a {spherical homogeneous space}. A {spherical embedding} of $G/H$ (or simply a spherical variety) is a normal variety $X$ equipped with a $G$-action so that there is an open dense $G$-orbit isomorphic to $G/H$.
\end{defi}

\subsubsection{Homogenous spherical datum}
Let $H$ be a spherical subgroup of $G$ with respect to the Borel subgroup $B$. The action of $G$ on the function field $\mathbb C(G/H)$ of $G/ H$ is given by
$$(g^*f)(x):=f(g^{-1}\cdot x),~\forall g\in G, x\in G/H~\text{and}~f\in\mathbb C(G/H).$$
A function $f(\not=0)\in\mathbb C(G/H)$ is called \emph{$B$-semiinvariant} if there is a character of $B$, denoted by $\chi$, so that $b^*f=\chi(b)f$ for any $b\in B$.
By \cite[Section 25.1]{Timashev-book}, $\mathbb C(G/H)^B=\mathbb C$. Two $B$-semiinvariant functions associated to a same character can differ from each other only by multiplying a non-zero constant.

Let $\mathfrak M(G/H)$ be the lattice of $  B$-characters which admits a correspondent $B$-semi-invariant functions, and $\mathfrak N(  G/  H)={\rm Hom}_\mathbb Z(\mathfrak M(  G/  H),\mathbb Z)$ its $\mathbb Z$-dual. There is a map $\varrho$ which maps a valuation $\nu$ of $\mathbb C(  G/  H)$ to an element $\varrho(\nu)$ in $\mathfrak N_\mathbb Q(G/H)=\mathfrak N(G/H)\otimes_\mathbb Z\mathbb Q$ by
\begin{align}\label{varrho-def}
\varrho(\nu)(\chi)=\nu(f),
\end{align}
where $f$ is a $B$-semiinvariant functions associated to $\chi$. Again, this is well-defined since $G/H$ is spherical with respect to $B$.
It is a fundamental result that $\varrho$ is injective on $  G$-invariant valuations and the image forms a convex cone (over $\mathbb Q$) $\mathcal V(  G/  H)$ in $\mathfrak N_\mathbb  Q(  G/  H)$, called the \emph{valuation cone} of $  G/  H$ (cf. \cite[Section 19]{Timashev-book}).
 Moreover, $\mathcal V(  G/  H)$ is a solid cosimplicial cone which is a (closed) fundamental chamber of a certain crystallographic reflection group, called the {\it little Weyl group} (denoted by $W^{G/H}$, cf. \cite[Sections 22]{Timashev-book}). In fact, $W^{G/H}$ is the Weyl group of the \emph{spherical root system} of $  G/  H$ (cf. \cite[Section 30]{Timashev-book}).
The spherical roots are defined to be the primitive generators of edges of $(-\mathcal V(G/H))^\vee$. The set of spherical roots is denoted by $\Pi_{G/H}$.

A $  B$-stable prime divisors in $  G/  H$ is called a \emph{colour}. Denote by $\mathcal D(  G/  H)$ the set of colours. A colour $D\in\mathcal D(  G/  H)$ also defines a valuation on $  G/  H$. However, the restriction of $\varrho$ on $\mathcal D(  G/  H)$ is usually non-injective.

Now we briefly introduce the homogeneous spherical datum, which by a deep result in \cite{Losev} (cf. \cite[Theorem 30.22]{Timashev-book}) characterizes the spherical homogeneous space. Let $P_\alpha$ be the minimal standard parabolic subgroup of $G$ containing $B$ corresponding to the simple root $\alpha\in\Pi_G$. Set $$\Pi_{G/H}^p:=\{\alpha\in\Pi_G|\mathcal D(G/H;\alpha)=\emptyset\}.$$
Then for every $\alpha\in\Pi_G\setminus\Pi_{G/H}$, it moves at least one colour. Set
$$\mathcal D(G/H;\alpha):=\{D\in\mathcal D(G/H)|D~\text{is not}~P_\alpha\text{-stable}\}.$$
Then $\mathcal D(G/H)=\cup_{\alpha\in\Pi_G} \mathcal D(G/H;\alpha)$. We see that a colour  $D\in\mathcal D(G/H)$ is of
\begin{itemize}
\item type a: if $D\in\mathcal D(G/H;\alpha)$ for $\alpha\in\Pi_{G/H}$;
\item type a': if $D\in\mathcal D(G/H;\alpha)$ for $2\alpha\in\Pi_{G/H}$;
\item type b: Otherwise.
\end{itemize}
Note that although a colour $D$ may belong to different $\mathcal D(G/H;\alpha)$, the type of $D$ is well-defined. Denote by $\mathcal D^\star(G/H)$ the colours in $\mathcal D(G/H)$ of type $\star$. Then $\mathcal D(G/H)=\mathcal D^a(G/H)\sqcup\mathcal D^{a'}(G/H)\sqcup\mathcal D^b(G/H)$. The set $\mathcal D^a(G/H)$ of type a colours if of special importance. Denote by $\varrho^a$ the restriction of $\varrho$ on $\mathcal D^a(G/H)$. We recall the following
\begin{defi}
The quadruple $(\mathfrak M(G/H), \Pi_{G/H}, \Pi_{G/H}^p, \mathcal D^a(G/H))$, together with the map $\varrho^a$, is called the homogeneous spherical datum of $G/H$.
\end{defi}
Here we regard $\mathcal D^a(G/H)$ as an abstract finite set equipped with a map $\varrho^a:\mathcal D^a(G/H)\to\mathfrak N(G/H)$. The homogeneous spherical datum was introduced by \cite{Luna}. It is proved by \cite{Losev} (see also \cite[Section 30]{Timashev-book}) that the homogeneous spherical datum uniquely determines $G/H$ up to $G$-equivariant isomorphism. In particular, the sets $\mathcal D^{a'}(G/H)$, $\mathcal D^b(G/H)$ and their images under $\varrho$ can also be recovered from the homogenous spherical datum (cf. \cite[Section 30]{Timashev-book} and \cite[Section 2]{Ga-Ho-datum}).

Given $G$ and a Borel subgroup $B$ of $G$. Given a sublattice $\bar{\mathfrak M}$ of $\mathfrak X(B)$, $\bar\Pi\subset\bar{\mathfrak M}$ a linearly independent set of primitive elements (spherical roots), $\bar\Pi^p\subset\Pi_G$, and $\bar{\mathcal D}^a$ an abstract finite set equipped with a map $\bar\varrho^a:\bar{\mathcal D}^a\to{\rm Hom}_\mathbb Z(\bar{\mathfrak M},\mathbb Z)$. The axioms that the abstract quadruple $(\bar{\mathfrak M}, \bar\Pi, \bar\Pi^p, \bar{\mathcal D}^a)$ together the map $\bar\varrho^a$ forms a homogeneous spherical datum can be found in \cite[Section 30]{Timashev-book}.

\subsubsection{The embedding theory}
It is a fundamental result that the spherical embeddings of a given $  G/  H$ are classified by combinatorial data called the \emph{coloured fan} \cite{Luna-Vust}. The equivariant embedding theorem is now referred as the Luna-Vust theory. We use \cite{Timashev-survey, Timashev-book} as main reference.
\begin{defi} Let $  H$ be a spherical subgroup of $  G$, $\mathcal D(  G/  H), \mathcal V(  G/  H)$ be the set of colours and the valuation cone, respectively.
\begin{itemize}
\item A {coloured cone} is a pair $(\mathfrak C,\mathfrak R)$, where $\mathfrak R \subset  \mathcal  D(  G/  H), O\not\in\varrho(\mathfrak R)$ and $\mathfrak C \subset \mathfrak N_\mathbb Q(  G/  H)$
 is a strictly convex cone generated by $\varrho(\mathfrak R)$ and finitely many elements of $\mathcal V(  G/  H)$ such
that the intersection of the relative interior of $\mathfrak C$ with $\mathcal V(  G/  H)$ is non-empty;
\item Given two coloured cones $(\mathfrak C,\mathfrak R)$ and $(\mathfrak C',\mathfrak R')$, We say that a coloured cone $(\mathfrak C',\mathfrak R')$ is a face
of another coloured cone $(\mathfrak C,\mathfrak R)$ if $\mathfrak C'$ is a face of $\mathfrak C$ and $\mathfrak R' = \mathfrak R \cap \varrho^{-1}(\mathfrak C')$;
\item A {coloured fan} is a collection $\mathfrak F$ of finitely many coloured cones such that the face of any its coloured cone is still in it, and any $v \in  \mathcal V(  G/  H)$ is contained in the relative interior of at most one of its cones.
\end{itemize}
\end{defi}

The $B$-charts (and hence the simple embeddings of $G/H$) are in bijection with coloured cones. Any  embedding is covered by finitely many simple embeddings, and the (complete) spherical embeddings are classified by coloured fans (cf. \cite[Section 15]{Timashev-book}),
\begin{thm}\label{coloured-fan-classification}
There is a bijection $X\leftrightarrow\mathfrak F_X$ between spherical embeddings of $  G/  H$ up to $  G$-equivariant isomorphism and coloured fans. There is a
bijection $Y \to (\mathfrak C_{Y},\mathfrak R_{Y})$ between the $  G$-orbits in $X$ and the coloured cones in $\mathfrak F_X$.
An orbit $Y$ is in the closure of another orbit $Z$ in $X$ if and only if the coloured cone $(\mathfrak{C}_Z,\mathfrak{R}_Z)$ is a face of $(\mathfrak C_Y,\mathfrak R_Y)$.
\end{thm}


\subsubsection{Line bundles and polytopes}
Let $X$ be a complete spherical variety, which is a spherical embedding of some $  G/  H$.
Let $L$ be a $  G$-linearlized line bundle on $X$.
In the following we will associated to $(X,L)$ several polytopes, which encode the geometric structure of $X$.

\subsubsection*{Polytope of a divisor}
Denote by $\mathcal I_{  G}(X)=\{D_A|A=1,...,d_0\}$ the set of $  G$-invariant prime divisors in $X$.
Then any $D_A\in\mathcal I_{  G}(X)$ corresponds to a $1$-dimensional cone $(\mathfrak C_A,\emptyset)\in\mathfrak F_X$.
Denote by $u_A$ the prime generator of $\mathfrak C_A$. Recall that $\mathcal D(  G/  H)$ is the set of colours, which are $  B$-stable but not $  G$-stable in $X$. Any $  B$-stable $\mathbb Q$-Weil divisor can be written as
\begin{align}\label{weil-div}
{\mathfrak d}=\sum_{A=1}^{d_0}\lambda_AD_A+\sum_{D\in\mathcal D(  G/  H)}\lambda_DD
\end{align}
for some $\lambda_A,\lambda_D\in\mathbb Q$. Set
$$\mathcal D_X:=\cup\{\mathfrak R\subset\mathcal D(  G/  H)|~\exists(\mathfrak C,\mathfrak R)\in\mathfrak F_X\}$$
By \cite[Proposition 3.1]{Brion89} (cf. \cite[Section 17]{Timashev-book}), ${\mathfrak d}$ is further a $\mathbb Q$-Cartier divisor if and only if there is a rational piecewise linear function $l_{\mathfrak d}(\cdot)$ on $|\mathfrak F_X|$ such that on each cone $\mathfrak C$ of $\mathfrak F_X$, $l_{\mathfrak d}(\cdot)=u_{\mathfrak C}(\cdot)$ for some $u_\mathfrak C\in\mathfrak M_\mathbb Q(G/H)$, and
\begin{align}\label{Cartier-condition}
\lambda_A=l_{\mathfrak d}(u_A),~A=1,...,d_0~\text{and}~\lambda_D=l_{\mathfrak d}(\varrho(D)),~\forall D\in\mathcal D_X.
\end{align}
It is further proved in \cite[Theorem 3.3]{Brion89} that when ${\mathfrak d}$ is ample, the function $l_{\mathfrak d}(\cdot)$ is strictly convex, and $\lambda_d>u_\mathfrak C(\varrho (D))$ for any $D\not\in\mathcal D_X$. Moreover, $l_{\mathfrak d}(x)=v_{\Delta_\mathfrak d}(-x)$ for $x\in|\mathfrak F_X|$, where $v_{\Delta_\mathfrak d}(\cdot)$ is the support function of some convex polytope $\Delta_{\mathfrak d}\subset\mathfrak M_\mathbb R(G/H)$. We call $\Delta_{\mathfrak d}$ the \emph{polytope of ${\mathfrak d}$}. In fact, $\Delta_\mathfrak d$ can be characterized by (cf. \cite{Brion89} and \cite[Section 17]{Timashev-book})
\begin{align*}
\{\lambda\in\mathfrak M_\mathbb R(G/H)|&\lambda_A+u_A(\lambda)\geq0,~A=1,...,d_0,
\\&\text{and}~\lambda_D+\varrho(D)(\lambda)\geq0,~\forall D\in\mathcal D(G/H)\}.
\end{align*}

\subsubsection*{Moment polytope of a line bundle}
Let $(X,L)$ be a polarized spherical variety. Then for any $k\in\mathbb N$ we can decompose ${\rm H}^0(X,L^k)$ as direct sum of irreducible $G$-representations,
\begin{align}\label{H0kL}
{\rm H}^0(X,L^k)=\oplus_{ \lambda\in \Delta_{L,k} } V_{ \lambda},
\end{align}
where $\Delta_{L,k}$ is a finite subset of the $  B$-characters $\mathfrak X(B)$ and each $  V_{ \lambda}$ is called an \emph{isotypic component}, which is isomorphic to the irreducible representation of $G$ with highest weight $ \lambda$.
Set
$$\Delta_L:=\overline{\cup_{k\in\mathbb N}(\frac1k\Delta_{L,k})}.$$
Then $\Delta_L$ is indeed a polytope in $\mathfrak X_\mathbb R(B)$. We call $\Delta_L$ the \emph{moment polytope of $(X,L)$} (cf. \cite[Section 17]{Timashev-book}). Clearly, the moment polytope of $(X,L^k)$ is $k$-times the moment polytope of $(X,L)$ for any $k\in\mathbb N_+$. Suppose that $s$ is a $  B$-semiinvariant (rational) section of $L$ with respect to a character $\chi$. Let ${\mathfrak d}$ be the divisor of $s$. Then $\Delta_{\mathfrak d}$ and $\Delta_L$ are related by
\begin{prop}\label{polytope-prop}(cf. \cite[Theorem 3.30]{Timashev-survey})
The two polytopes $\Delta_L$ and $\Delta_{\mathfrak d}$ are related by $$\Delta_L=\chi+\Delta_{\mathfrak d}.$$
\end{prop}

Thus, if we fix any $B$-semiinvariant section $s_0$ of $L$ with $B$-weight $\chi_0$,
then $\Delta_L-\chi_0$ is indeed a polytope in $\mathfrak M_\mathbb R( G/ H)$. Also, it holds
\begin{align}\label{Delta-L-k}
\Delta_{L,k}=k\Delta_L\cap(\mathfrak M(G/H)+k\chi_0),~\forall k\in\mathbb N.
\end{align}
 Note that when a finite covering of $G/H$ is $\mathbb Q$-factorial, or $G/H$ is the flat limit of some space with $\mathbb Q$-factorial finite covering, then we can choose $\chi_0\in\mathfrak M_\mathbb R( G/ H)$ and therefore $\Delta_L\subset\mathfrak M_\mathbb R( G/ H)$.

When $X$ is ($\mathbb Q$-)Fano and $L=-K_X$, \cite{Ga-Ho} proved that there is a canonical choice of $B$-invariant divisor ${\mathfrak d_0}$ of $L$ in form of \eqref{weil-div},
\begin{align}\label{-K-d0}
{\mathfrak d_0}=\sum_{A=1}^{d_0}D_A+\sum_{D\in\mathcal D(  G/  H)}m_DD,
\end{align}
where the coefficients $m_D$'s are explicitly obtained in \cite{Ga-Ho} according to the type of each colour $D$. In fact, this divisor corresponds to a $B$-semiinvariant section $\mathfrak d_0$ of
$-K_X$ (in case $X$ is Gorenstein Fano) with $B$-weight
\begin{align}\label{kappa}
\kappa_P=\sum_{\alpha\in\Phi^G_+,\alpha\not\perp\Delta_L}\alpha.
\end{align}
Here $\Phi^G_+$ are the set of positive roots of $G$ with respect to $B$, and the orthogonal is considered with respect to the Killing form. Recall the polytope $\Delta_{\mathfrak d_0}$ of $\mathfrak d_0$. Its dual polytope $\Delta_{\mathfrak d_0}^*$ is a ($\mathbb Q$-)$G/H$-reflexive polytopes defined in \cite[Section 7]{Ga-Ho}, which is the convex hull
$$\Delta_{\mathfrak d_0}^*={\rm Conv}(\{D_A|A=1,...,d_0\}\cup\{\varrho(D)/m_D|D\in\mathcal D(G/H)\}).$$
Moreover, the coloured fan $\mathfrak F_X$ can be recovered from $\Delta^*_{\mathfrak d}$ by taking its supported coloured face fan (cf. \cite[Sections 7-8]{Ga-Ho}). 

\subsubsection{The Okounkov body of a spherical variety}
For each dominant weight $ \lambda$ of $  G$, there is a Gel'fand-Tsetlin polytope $\Delta^{\rm GT}( \lambda)$ which has the same dimension with the maximal unipotent subgroup $  N_u$ of $  G$ (cf. \cite{Kirichenko-Smirnov-Timorin}). It is known that
\begin{align}\label{okounkov-body-fibre}
\dim(V_{ \lambda})=\text{number of integral points in $\Delta^{\rm GT}( \lambda)$}.
\end{align}

Let $X$ be a $  G$-spherical variety. It is proved in \cite{Okounkov} that the Okounkov body $\Delta^{\rm O}$ of $X$ is given by the convex hull
\begin{align}\label{okounkov-body}
\Delta^{\rm O}&:={\rm Conv}\left(\cup_{k\in\mathbb N_+}\cup_{ \lambda\in{\Delta_L}\cap\frac1k\Delta_{L,k}}( \lambda,\frac1k\Delta^{\rm GT}(k \lambda))\right)\notag\\
&\subset\Delta_L\times\mathbb R^{\dim(  N_u)}.
\end{align}
Note that the Gel'fand-Tsetlin polytope $\Delta^{\rm GT}( \lambda)$ is linear in $ \lambda$.
Thus $\Delta^{\rm O}$ is a convex polytope in $(\mathfrak M_\mathbb R(  G/  H)+\chi_0)\times\mathbb R^{\dim(  N_u)}.$
It is a fibration over $\Delta_L$ so that the fibre at each $ \lambda\in\frac1k\Delta_{L,k}$ is $\frac1k\Delta^{\rm GT}(k \lambda)$.

\subsection*{Notations}
In this paper we use the following conventions:
\begin{itemize}
\item $\Phi^\cdot$-a root system;
\item $\Phi_{+}^{\cdot}$-a chosen set of positive roots in $\Phi^\cdot$;
\item $\Pi_{\cdot}$-the simple roots in $\Phi_{+}^{\cdot}$;
\item $\mathfrak X(\cdot)$-characters of a group;
\item $\mathfrak X_+(\cdot)$-dominant weights of a group;
\item ${\rm Conv}(\cdot)$-convex hull of a set;
\item ${\rm RelInt}(\cdot)$-relative interior of a set;
\item $S^G$-the subset of all $G$-invariants in a set $S$ with group $G$-action;
\item $S^{(H)}_\chi$-the subset of all $H$-semiinvariants with $H$-character $\chi$ in a set $S$ admitting a group $H$-action.
\end{itemize}

\section{Classification of $G$-equivariant normal $\mathbb R$-test configurations}

\subsection{Multiplication in the Kodaira ring}
Suppose that $G/H$ is a spherical homogeneous space and $L$ is a $G$-linearized line bundle on it. For example, we may choose any projective $G$-spherical embedding $X$ of $G/H$, which is embedded in projective space by some very ample line bundle $L_X$, and take $L=L_X|_{G/H}$. Denote by $\chi$ the $H$-character so that $H$ acts on $L_{eH}$ through $\chi$. For any $k\in\mathbb N_+$, we can decompose ${\rm H}^0(G/H,L^k)$ into direct sum of isotypic components
\begin{align}\label{H0(G/H-Lk)}
{\rm H}^0(G/H,L^k)=\oplus_{\lambda\in\Delta_{L,k}}V_\lambda,~\text{for a finite set $\Delta_{L,k}\subset\mathfrak X_+(G)$.}
\end{align}
Here by $\mathfrak X_+(G)$ we denote the set of dominant $G$-weights.

To investigate the multiplication in the Kodaira ring $$R(X,L)=\oplus_{k\in\mathbb N}{\rm H}^0(G/H,L^k),$$ we first associate each isotypic component $V_\lambda$ in \eqref{H0(G/H-Lk)} an isotypic component in $\mathbb C[G]$. Recall that by \cite[Section 2]{Timashev-book}, the restriction $L$ on $G/H$ is isomorphic to some homogenous line bundle $G*_H\mathbb C_\chi$, where $\mathbb C_\chi\cong\mathbb C$ with the $H$-action via a character $\chi$ of $H$. On the other hand, consider the pull back $\pi_{G/H}^*L$ of $L$ on $G$ via the canonical projection $\pi_{G/H}:G\to G/H$. Since $G$ is affine,\footnote{Up to replacing $G$ by its finite covering we can always assume that $G$ is affine.} $\pi_{G/H}^*L$ is trivial. It suffices to find in
\begin{align}\label{C[G]}
\mathbb C[G]=\oplus_{\lambda\in\mathfrak X_+(G)}V_\lambda\otimes V_\lambda^*
\end{align}
the $H$-semiinvariants with respect to the $H$-character $k\chi$, that descends to a section in ${\rm H}^0(G/H,L^k)$. Here $H$ acts on the $V_\lambda^*$-factor from right. More precisely, we get
\begin{align}\label{Kodaira-G/H-L}
R^o:=\oplus_{k\in\mathbb N}{\rm H}^0(G/H,L^k)=\oplus_{k\in\mathbb N}\oplus_{\lambda\in\mathfrak X_+(G)}V_\lambda\otimes (V_\lambda^*)^{(H)}_{k\chi}.
\end{align}
Thus, $\lambda\in\Delta_{L,k}$ for some $k\in\mathbb N$ if and only if $(V_\lambda^*)^{(H)}_{k\chi}\not=\{O\}$. Note that by \cite[Theorem 25.1]{Timashev-book}, for any $\lambda\in\mathfrak X_+(G)$ and $k\in\mathbb N$,
$$\dim((V_\lambda^*)^{(H)}_{k\chi})\leq1.$$
Hence $V_\lambda\subset \oplus_{k\in\mathbb N}{\rm H}^0(G/H,L^k)$ if and only if
\begin{align}\label{V-lambda-in-H0}
\lambda\in\mathfrak X_+(G)_{k\chi}:=\{\lambda\in\mathfrak X_+(G)|\dim((V_\lambda^*)^{(H)}_{k\chi})=1\}.
\end{align}

Then we have:
\begin{prop}\label{mult-rule-isotypic}
Suppose that $L$ is the very ample line bundle as above. Then for any $k_1,k_2\in\mathbb N_+$ and $\lambda_i\in\Delta_{L,k_i},~i=1,2$. We can decompose $V_{\lambda_1}\cdot V_{\lambda_2}$ into direct sum of isotypic components
\begin{align*}
V_{\lambda_1}\cdot V_{\lambda_2}=\oplus_{\mu}V_{\lambda_1+\lambda_2-\mu},
\end{align*}
where $\mu$ is summing over all non-negative linear combinations of simple roots of $G$ satisfies the following Conditions ($\star$1)-($\star$2):
\begin{itemize}
\item[($\star$1)] $\mu$ appears in
\begin{align}\label{tensor-tail}
V_{\lambda_1}\otimes V_{\lambda_2}=\oplus_\mu V_{\lambda_1+\lambda_2-\mu}^{\oplus m_\mu},~m_\mu\in\mathbb N_+,
\end{align}
so that $(\lambda_1+\lambda_2-\mu)\in\mathfrak X_+(G)_{(k_1+k_2)\chi};$
\end{itemize}
and
\begin{itemize}
\item[($\star$2)] $\mu$ satisfies the condition that
$(V_{\lambda_1}^*)^{(H)}_{k_1\chi}\otimes(V_{\lambda_2}^*)^{(H)}_{k_2\chi}$ have nonzero component in $(V_{\lambda_1+\lambda_2-\mu}^*)^{(H)}_{(k_1+k_2)\chi}$

\end{itemize}
\end{prop}

\begin{proof}

Given $G$-module $V_\lambda$, by taking some basis of $V_\lambda^*$, we can get a representation of $V_\lambda^*$: $\rho_\lambda: G\to {\rm GL}_n(C)$ where $n=\dim V_\lambda$. Then $V_\lambda\otimes V_\lambda^*$ can be realized as a linear space spanned by the functions $\rho_{\lambda,i,j}: G\to \mathbb C$ which is given by taking the $(i,j)$-entry of the matrix $\rho_\lambda(g),g\in G$. $(V_{\lambda_1}\otimes V_{\lambda_1}^*)\cdot(V_{\lambda_2}\otimes V_{\lambda_2}^*)$ in $\mathbb C[G]$ is the linear space spanned by the functions by taking entries of a representation of the $G$-module $V_{\lambda_1}^*\otimes V_{\lambda_2}^*$.

Consider $V_{\lambda_1}^*\otimes V_{\lambda_2}^*=\oplus_\mu V_{\lambda_1+\lambda_2-\mu}^*$ ($\mu$ may appear with multiplicity), then taking a basis $v_{1,\mu}, \dots, v_{l_{\mu},\mu}$ for each $V_{\lambda_1+\lambda_2-\mu}^*$ where $l_{\mu}=\dim V_{\lambda_1+\lambda_2-\mu}$ such that $v_{1,\mu}\in (V_{\lambda_1+\lambda_2-\mu}^*)^{(H)}_{(k_1+k_2)\chi}$ if $(V_{\lambda_1+\lambda_2-\mu}^*)^{(H)}_{(k_1+k_2)\chi}\ne 0.$ Under each basis, we get a representation $\rho_{\lambda_1+\lambda_2-\mu}: G\to {\rm GL}_{l_\lambda}(\mathbb C)$ and $V_{\lambda_1+\lambda_2-\mu}\otimes (V_{\lambda_1+\lambda_2-\mu}^*)^{(H)}_{(k_1+k_2)\chi}$ can be realised the linear space spanned by the functions $\rho_{\lambda_1+\lambda_2-\mu,1,k}$ where $k=1,\dots,l_{\mu}$ if $(V_{\lambda_1+\lambda_2-\mu}^*)^{(H)}_{(k_1+k_2)\chi}\ne 0$. The union of such basis forms a basis of $V_{\lambda_1}^*\otimes V_{\lambda_2}^*$.

Now consider $(V_{\lambda_1}\otimes (V_{\lambda_1}^*)^{(H)}_{k_1\chi})\cdot(V_{\lambda_2}\otimes (V_{\lambda_2}^*)^{(H)}_{k_2\chi})$. Given $0\not= v_{\lambda_1,\lambda_2}\in (V_{\lambda_1}^*)^{(H)}_{k_1\chi}\otimes(V_{\lambda_2}^*)^{(H)}_{k_2\chi}$ and taking a basis  of $V_{\lambda_1}^*\otimes V_{\lambda_2}^*$ such that $v_{\lambda_1,\lambda_2}$ is the first vector of this basis. Under such basis, there is a representation $\rho_{\lambda_1,\lambda_2}$ of $V_{\lambda_1}^*\otimes V_{\lambda_2}^*$, then $(V_{\lambda_1}\otimes (V_{\lambda_1}^*)^{(H)}_{k_1\chi})(V_{\lambda_2}\otimes (V_{\lambda_2}^*)^{(H)}_{k_2\chi})$ is spanned by the functions $\rho_{\lambda_1,\lambda_2,1,k}$ where $k=1,\dots, \dim ( V_{\lambda_1}\otimes V_{\lambda_2})$.

Since $$(V_{\lambda_1}^*)^{(H)}_{k_1\chi}\otimes(V_{\lambda_2}^*)^{(H)}_{k_2\chi} \subseteq (V_{\lambda_1}^*\otimes V_{\lambda_2}^*)^{(H)}_{(k_1+k_2)\chi}=\oplus_\mu(V_{\lambda_1+\lambda_2-\mu}^*)^{(H)}_{(k_1+k_2)\chi},$$ then $v_{\lambda_1,\lambda_2}=\sum_{\mu}a_{\mu}v_{1,\mu}$. By the above, we see that $$(V_{\lambda_1}\otimes (V_{\lambda_1}^*)^{(H)}_{k_1\chi})\cdot(V_{\lambda_2}\otimes (V_{\lambda_2}^*)^{(H)}_{k_2\chi})$$ is spanned by the functions $a_\mu\rho_{\lambda_1+\lambda_2-\mu,1,k}$ where $k=1,\dots, l_\mu$ and $\mu$ satisfies the condition ($\star1$). Thus $$(V_{\lambda_1}\otimes (V_{\lambda_1}^*)^{(H)}_{k_1\chi})(V_{\lambda_2}\otimes (V_{\lambda_2}^*)^{(H)}_{k_2\chi})=\oplus_{a_{\mu}\ne 0 }V_{\lambda_1+\lambda_2-\mu}\otimes (V_{\lambda_1+\lambda_2-\mu}^*)^{(H)}_{(k_1+k_2)\chi},$$ whence the result.

\end{proof}

Consider the isotypic decomposition of any $p$ isotypic components
\begin{align}\label{v1-cdot-cdot-Vp}
V_{\lambda_1}\cdot...\cdot V_{\lambda_p}=\oplus_{\mu}V_{\lambda_1+...+\lambda_p-\mu},
\end{align}
in Proposition \ref{mult-rule-isotypic}, where each $\lambda_i\in \Delta_{L,k_i}(\subset\mathfrak X_+(G))$ and $\mu$ is a non-negative rational linear combination of roots in $\Pi_G$.
Set
\begin{align*}
\mathfrak K:=\{\mu|&\text{There are $\lambda_i\in\Delta_{L,k_i},~k_i\in\mathbb N$ for $i=1,..,p$ so that} \\&\text{$V_{\lambda_1+...+\lambda_p-\mu}$ appears in \eqref{v1-cdot-cdot-Vp}}\}.
\end{align*}

Using the argument of \cite[Section 20]{Timashev-book}, we have
\begin{coro}\label{val-cone}
The set $\mathfrak K$ is a finitely generated semigroup and $\mathcal V(G/H)$ is minus the dual cone of the cone generated by $\mathfrak K$. Moreover, let $\{\mu_j\}$ be the set of generators of ${\rm Span}_{\mathbb Q_{\geq0}}(\mathfrak K)$. Then up to rescaling each $\mu_{j}$ by a positive rational number, $\{\mu_{j}\}$ coincides with the set of spherical roots $\Pi_{G/H}$.
\end{coro}

\begin{proof}
This proof in fact goes back to \cite[Section 3]{Timashev-survey} when $G/H$ is quasiaffine and \cite[Section 20]{Timashev-book} for general cases. We sketch it here for readers' convenience. First, we deal the case when $G/H$ is quasiaffine. In this case $\chi=0$ and $R^o=\mathbb C[G/H]$ in \eqref{Kodaira-G/H-L}. Also, up to a sign difference, $\mathfrak K$ is precisely the set of tail vectors of $R^o$ defined in \cite[Definition 20.9]{Timashev-book}. It is proved by \cite[Proposition 2.13]{Ale-Bri} (see also \cite[Corollary E.15]{Timashev-book}) that $\mathfrak K$ is finitely generated. By \cite[Section 20.5]{Timashev-book}, $\mathfrak v\in\mathcal V(G/H)$ if and only if it is non-positive on the tail vectors. Thus $\mathcal V(G/H)=-(\mathfrak K)^\vee$. Recall that the spherical roots are the primitive generators of $-(\mathcal V(G/H))^\vee$. We get the Corollary when $G/H$ is quasiaffine.

General cases can be reduced to quasiaffine cases by \cite[Remark 20.8]{Timashev-book}. We want to construct some function ring on $G/H$ so that its field of quotients is $\mathbb C(G/H)$. Let $Z$ be any projective spherical embedding of $G/H$, which is embedded in projective space via some very ample line bundle $L$, and consider the affine cone $\mathfrak C(Z)$ of $Z$. Then $\mathfrak C(Z)$ is an affine $G\times \mathbb C^*$-spherical homogeneous space. As pointed out in \cite[Remark 20.8]{Timashev-book}, the valuation cone $\mathcal V(\mathfrak C(Z))$ projects to $\mathcal V(G/H)$, and the grading of $\mathbb C[\mathfrak C(Z)]$ determines two $G\times\mathbb C^*$-invariant valuations which precisely span the kernel of this projection. Let $R(Z,L)=\oplus_{k\in\mathbb N}{\rm H}^0(Z,L^k)$
be the Kodaira ring of $(Z,L)$. Up to replace $L$ by $L^d$ for some sufficiently large integer $d$, we may assume that $R(Z,L)$ is generated by the piece of $k=1$. 
Then $\mathbb C[\mathfrak C(Z)]=\oplus_{k\in\mathbb N}\mathbb C[\mathfrak C(Z)]^{(\mathbb C^*)}_k$, and each $\mathbb C[\mathfrak C(Z)]^{(\mathbb C^*)}_k$ is the pull back of ${\rm H}^0(Z,L^k)$ on $\mathfrak C(Z)$ via the canonical projection.

On the other hand, let $\chi$ be the $H$-character of $L$ in \eqref{Kodaira-G/H-L}. Denote by $\mathcal O$ the open $G\times\mathbb C^*$-orbit in $\mathfrak C(Z)$. Then $\mathcal O$ is quasiaffine, and $\mathbb C(\mathcal O)=\mathbb C(\mathfrak C(Z))$. In fact, $\mathcal O=G\times\mathbb C^*/\tilde H$, where
$$\tilde H=\{(h,\chi^{-1}(h))|h\in H\}\subset H\times\mathbb C^*.$$
Since every function in $\mathbb C[\mathcal O]$ pulls back to a function in $\mathbb C[G\times\mathbb C^*]$, we have
\begin{align}\label{C[O]}
\mathbb C[\mathcal O]=\oplus_{(\lambda,k)\in\mathfrak X_+(G)\oplus\mathbb Z}\tilde V_{(\lambda,k)}\otimes (\tilde V^*_{(\lambda,k)})^{\tilde H},
\end{align}
where $\tilde V_{(\lambda,k)}\otimes\tilde V^*_{(\lambda,k)}$ is the $(\lambda,k)$-isotypic component of $\mathbb C[G\times\mathbb C^*]$. It is direct to check that a $\tilde V_{(\lambda,k)}\otimes (\tilde V^*_{(\lambda,k)})^{\tilde H}$-factor in \eqref{C[O]} is non-trivial if and only if the $V_\lambda\otimes (V_\lambda^*)^{(H)}_{k\chi}$-factor in \eqref{Kodaira-G/H-L} is non-trivial. 
Also, from Proposition \ref{mult-rule-isotypic} we have
$$\tilde V_{(\lambda_1,k_1)}\cdot\tilde V_{(\lambda_2,k_2)}=\oplus_{\mu~\text{satisfies the Conditions ($\star$1)-($\star$2)}}\tilde V_{(\lambda_1+\lambda_2-\mu, k_1+k_2)}.$$
Thus $\mathfrak K$ is the tail vectors of $\mathbb C[\mathcal O]$, and is finitely generated. By \cite[Corollary 19.6]{Timashev-book}, taking projection to the $\mathfrak N_\mathbb Q(G/H)$-factor we get the remaining parts of the Corollary.
\end{proof}

\subsection{The classification theorem}

With the help of Proposition \ref{mult-rule-isotypic}, we can show the main classification results Theorems \ref{main-thm-1}-\ref{main-thm-3}. The remaining part of the proof is essentially from \cite{LL-arXiv-2021}.

Suppose that $G/H$ is a spherical homogeneous space and $(X,L)$ a polarized spherical embedding of $G/H$. For simplicity, we write $\mathfrak M$ in short of $\mathfrak M_\mathbb R(G/H)$.
Also assume that there is a $B$-semiinvariant section $s_0$ of $L$ with weight $\chi_0$. Recall \eqref{Delta-L-k}, we can decompose ${\rm H}^0(X,L^k)$ into direct sum of isotypic components (cf. \cite[Section 17.4]{Timashev-book}),
\begin{align}\label{H0-M}
R_k={\rm H}^0(X,L^k)=\oplus_{\lambda\in  \Delta_{L,k}}V_\lambda,
\end{align}
where $V_\lambda$ is the irreducible $G$-representation of highest weight $\lambda$. The Kodaira ring of $M$ is given by
\begin{align}\label{R}
R(X,L)=\oplus_{k\in\mathbb N}R_k.
\end{align}

Suppose that $\mathcal F$ is a $G$-equivariant filtration on $R$. Then by Definition \ref{filtrantion-def} (1)-(2), we have
\begin{align}\label{F-s-k}
\mathcal F^sR_k=\oplus_{s^{(k)}_\lambda\geq s}V_\lambda,
\end{align}
where we associated to each $V_\lambda$ in \eqref{H0-M} a number $s^{(k)}_\lambda$. From \eqref{H0-M} we see that the filtration $\mathcal F$ is totally determined by the collection of $s_{\lambda}^{(k)}$'s. Recall the Abelian group $\Gamma(\mathcal F)$ defined after \eqref{group-discon}. Under the normalization of Remark \ref{F-normalized}, we see that the corresponding Rees algebra \eqref{rees-def} reduces to
\begin{align}\label{rees}
{\rm R}(\mathcal F)=\oplus_{k\in\mathbb N}\oplus_{s\in\Gamma(\mathcal F)}\oplus_{\lambda\in  \Delta_{L,k},s^{(k)}_\lambda\geq s}t^{-s}V_\lambda.
\end{align}
$\mathcal F$ is an $\mathbb R$-test configuration if and only if \eqref{rees} is finitely generated.

For any two $\lambda_i\in \Delta_{L,k_i}$ and $s_i$ such that $V_{\lambda_i}\subset\mathcal F^{s_i}R_{k_i}$, $i=1,2$. By Definition \ref{filtrantion-def} (4),
$$V_{\lambda_1+\lambda_2}\subset V_{\lambda_1}\cdot V_{\lambda_2}\subset\mathcal F^{s_1}R_{k_1}\cdot\mathcal F^{s_2}R_{k_2}\subset\mathcal F^{s_1+s_2}R_{k_1+k_2}.$$
Thus, we have $s_{\lambda_1+\lambda_2}^{(k_1+k_2)}\geq s_1+s_2.$ In particular,
\begin{align}\label{alomost-concave}
s_{\lambda_1+\lambda_2}^{(k_1+k_2)}\geq s_{\lambda_1}^{(k_1)}+s_{\lambda_2}^{(k_2)}.
\end{align}

Also, Definition \ref{filtrantion-def} (3) is equivalent to that $\{s^{(k)}_\lambda/k\}_{\lambda\in \Delta_{L,k}}$ are uniformly bounded with respect to all $k\in\mathbb N_+$. 

Denote by $\mathcal V_\mathbb R(G/H)$ the closure of $\mathcal V(G/H)$ in $\mathfrak N_\mathbb R(G/H)$. Assume that $\Gamma(\mathcal F)$ is normalized according to Remark \ref{F-normalized}. We have
\begin{thm}\label{G-classify}
Let $(X,L)$ be a polarized spherical embedding of $G/H$ with moment polytope $\Delta_L$. Then for any $G$-equivariant normal $\mathbb R$-test configuration $\mathcal F$ of $(X,L)$, there is a concave, piecewise-linear function $f$ on $\Delta_L$ whose domains of linearity consist of rational polytopes in $\mathfrak M_\mathbb Q$ such that $\min f=0$, $\nabla f\in\mathcal V_\mathbb R(G/H)$ and
\begin{align}\label{s-k-G-final}
s_{\lambda}^{(k)}=\max\{s\in\Gamma(\mathcal F)|s\leq kf(\frac\lambda k)\},~\forall \lambda\in \Delta_{L,k}~\text{and}~k\in\mathbb N.
\end{align}
Moreover, $s_{\lambda}^{(k)}=kf(\frac\lambda k)$ if $\frac1k\lambda$ is a vertex of the domains of linearity of $f$.

Conversely, given any such $f$ and $r_0\in \mathbb N_+$ so that the domains of linearity of $r_0f(\frac\cdot {r_0})$ in $r_0\Delta_L$ consist of integral polytopes in $\mathfrak M+r_0\chi_0$. Denote by $\Gamma_{r_0}({\rm Vert}(f))$ the Abelian group generated by
$$\{r_0f(\frac1r_0\lambda)|\lambda~\text{is a vertex of a domain of linearity of $f$}\}.$$
Then for any finitely generated Abelian group $\Gamma$ containing $\Gamma_{r_0}({\rm Vert}(f))$, the collection of points of discontinuity
\begin{align}\label{s-k-G-general-inv}
s_{\lambda}^{(k)}=\sup\{s\in\Gamma|s\leq kf(\frac\lambda k)\},~\forall \lambda\in \Delta_{L,k}~\text{and}~k\in\mathbb N.
\end{align}
together with \eqref{F-s-k} defines a $G$-equivariant normal $\mathbb R$-test configuration $\mathcal F$ of $(X,L)$ satisfying  \eqref{s-k-G-final}.
\end{thm}


\begin{proof}
The proof will be divided into two parts according to the two directions.

\textbf{Part-1: From equivariant normal $\mathbb R$-test configurations to \eqref{s-k-G-final}.}

Suppose that $\mathcal F$ is a $G$-equivariant normal $\mathbb R$-test configuration of $(X,L)$. By \eqref{total-space-ring}, ${\rm R}(\mathcal F)$ is the Kodaira ring of the total space. Then up to replacing $L$ by some $L^{r_0}$ with $r_0\in\mathbb N_+$, we may assume that the Rees algebra ${\rm R}(\mathcal F)$ in \eqref{rees} is a normal ring.

We are going to construct a function $f$ satisfying \eqref{s-k-G-final}. As in \cite[Proposition 2.15]{Boucksom-Hisamoto-Jonsson}, for sufficiently large $r_0\in\mathbb N_+$ we may assume that the Rees algebra \eqref{rees} is generated by the piece $k=1$. We can choose $r_0$ sufficiently divisible so that even each vertex of $  {r_0\Delta_L}$ is also integral. Also, without loss of generality, we can assume that $s_\lambda^{(1)}\geq0$ for all $\lambda\in \Delta_{L,1}(=\Delta_L\cap(\mathfrak M+\chi_0))$.

Let $\lambda\in \Delta_{L,k}(=k \Delta_L\cap(\mathfrak M+k\chi_0))$ and $\mu_1,...,\mu_k\in\lambda\in \Delta_{L,1}$ so that
$$V_\lambda\subset V_{\mu_1}\cdot...\cdot V_{\mu_k}.$$
By \eqref{alomost-concave} we have
$$s_\lambda^{(k)}\geq\sum_{j=1}^k s_{\mu_j}^{(1)}.$$
Since \eqref{rees} is generated by the piece $k=1$, we get for any $k\in\mathbb N_+$ and $\lambda\in \Delta_{L,k}$,
\begin{align}\label{s-k-G}
s^{(k)}_\lambda=\max\{\sum_{i=1}^ks^{(1)}_{\mu_i}|&\mu_i\in\Delta_{L,1}, V_\lambda\subset V_{\mu_1}\cdot...\cdot V_{\mu_k}\}\geq0.
\end{align}

\emph{Step-1.1: Comparison of points of discontinuity.} Let $\Pi_{G/H}=\{\alpha_1,...,\alpha_r\}$ be the  spherical roots. Then each $\alpha_i$ is a non-negative integral linear combination of simple roots in $\Pi_G$. We first show that for any $\lambda,\mu\in \Delta_{L,1}$ satisfying
\begin{align}\label{point-compare}
\lambda=\mu-\sum_{i=1}^rc_i\alpha_i,~0\leq c_i\in\mathbb Q~\text{for all}~i=1,...,r,
\end{align}
it holds
\begin{align}\label{s-compare}
s^{(1)}_\lambda\geq s^{(1)}_\mu.
\end{align}

Otherwise, if $s^{(1)}_\lambda<s^{(1)}_\mu$, by \eqref{F-s-k} we see that
\begin{align}\label{not-in-R}
t^{-s_\mu^{(1)}} V_\lambda\not\subset {\rm R}(\mathcal F).
\end{align}
Let $e_\lambda$ be the highest weight vector in $V_\lambda$. We see that for any $k\in\mathbb N_+$,
\begin{align}\label{power-in}
(t^{-s_\mu^{(1)}}e_\lambda)^{\cdot k}\in t^{-ks_\mu^{(1)}} (V_{k\lambda}).
\end{align}
Note that by \eqref{point-compare},$$\lambda\in{\rm Conv}\{w(\mu)|w\in W^G\}.$$
Hence by \cite[Lemma 1]{Timashev-Sbo}, there is a $k_0\in\mathbb N_+$ such that $V_{k_0\lambda}\subset V_\mu^{\otimes k_0}$, and consequently
$$V_{k_0m\lambda}\subset V_{k_0\mu}^{\otimes m}\subset V_\mu^{\otimes k_0m},~\forall m\in\mathbb N_+.$$
On the other hand, denote by $\chi$ the character of the $H$-action on $L_{eH}$. Then by Corollary \ref{val-cone}, there is an $n_0\in\mathbb N_+$ so that
$$n_0(\mu-\lambda)\in\mathfrak X_+(G)_{\mathbb N\chi}.$$
Combining with Proposition \ref{mult-rule-isotypic}, we have
$$V_{k_0n_0\lambda}\subset V_{\mu}^{\cdot k_0n_0}.$$
We get from \eqref{power-in} that
$$(t^{-s_\mu^{(1)}}e_\lambda)^{\cdot k_0n_0}\in (t^{-s_\mu^{(1)}} V_{\mu})^{\cdot k_0n_0}\subset {\rm R}(\mathcal F).$$
Thus $t^{-s_\mu^{(1)}}e_\lambda$ is integral over ${\rm R}(\mathcal F)$. Since ${\rm R}(\mathcal F)$ is normal, $$t^{-s_\mu^{(1)}}e_\lambda\in {\rm R}(\mathcal F),$$
which contradicts to \eqref{not-in-R} and we conclude \eqref{s-compare}.

\emph{Step-1.2: Construction of $f$.} In view of \eqref{point-compare}, for simplicity we will write ``$\mu\succeq\lambda$" whenever $\lambda$ and $\mu$ satisfy \eqref{point-compare}.
We claim that for each $k\in\mathbb N_+$, $\mu,\lambda_1,...,\lambda_l\in \Delta_{L,k}$ and constants $0\leq c_1,...,c_l\leq1$ satisfying
\begin{align}
&\mu=\sum_{i=1}^lc_i\lambda_i,\label{mu-sum-lambda}
\end{align}
and
\begin{align}
&\sum_{1=1}^lc_i=1,\label{sum-ci}
\end{align}
it always holds
\begin{align}\label{s-k-concave}
s^{(k)}_\mu\geq\sup\{s\in\Gamma(\mathcal F)|s\leq\sum_{i=1}^lc_is^{(k)}_{\lambda_i}\}.
\end{align}

Otherwise, suppose that there are $\mu,\lambda_1,...,\lambda_l\in \Delta_{L,k}$ and $0\leq c_1,...,c_l\leq1$ satisfying satisfying \eqref{mu-sum-lambda}-\eqref{sum-ci} but \eqref{s-k-concave} fails. Then
\begin{align}\label{interpolation}
s^{(k)}_\mu<\sum_{i=1}^lc_i\hat s^{(k)}_{\lambda_i}=:\hat s^{(k)}_\mu\in\Gamma(\mathcal F),
\end{align}
where we can choose $\hat s^{(k)}_{\lambda_i}\in\Gamma(\mathcal F)_\mathbb Q$ with $\hat s^{(k)}_{\lambda_i}\leq s^{(k)}_{\lambda_i}.$ In fact, given any $\hat s^{(k)}_\mu\in\Gamma(\mathcal F)$ so that
$$s^{(k)}_\mu<\hat s^{(k)}_\mu\leq\sum_{i=1}^lc_is^{(k)}_{\lambda_i},$$
we can choose
$$\hat s^{(k)}_{\lambda_i}=s^{(k)}_{\lambda_i}-(\sum_{i=1}^lc_is^{(k)}_{\lambda_i}-\hat s^{(k)}_\mu)$$
to satisfy \eqref{interpolation}. Let $e_\mu\in V_\mu$ be the highest weight vector, as in \eqref{not-in-R}, we have
\begin{align}\label{contra-1}
t^{-\hat s^{(k)}_\mu}e_\mu\not\in {\rm R}(\mathcal F).
\end{align}
On the other hand, we can choose $n_0,m_0\in\mathbb N_+$ so that $n_0c_j\in\mathbb N$ and $m_0\hat s_{\lambda_i}^{(k)}\in\Gamma(\mathcal F)$ for $j=1,...,l$. Then
\begin{align*}
(t^{-\hat s^{(k)}_\mu}e_\mu)^{\cdot n_0m_0}=&t^{-\sum_{j=1}^ln_0m_0c_j\hat s^{(k)}_{\lambda_j}}(e_\mu)^{\cdot n_0m_0}\in t^{-\sum_{j=1}^ln_0m_0c_j\hat s^{(k)}_{\lambda_j}}  V_{n_0m_0\mu}.
\end{align*}
However, by \eqref{mu-sum-lambda}-\eqref{sum-ci},
\begin{align}\label{t-sum-hat-s}
&t^{-\sum_{j=1}^ln_0m_0c_j\hat s^{(k)}_{\lambda_j}}V_{n_0m_0\mu}\notag\\
\subset& (t^{-m_0\hat s^{(k)}_{\lambda_1}} (V_{m_0\lambda_1}))^{\cdot n_0c_1}\cdot...\cdot(t^{-m_0\hat s^{(k)}_{\lambda_l}} (V_{m_0\lambda_l}))^{\cdot n_0c_l}.
\end{align}
Since each $\hat s_{\lambda_i}^{(k)}m_0\leq s_{\lambda_i}^{(k)}m_0\leq s_{m_0\lambda_i}^{(m_0k)}$, by \eqref{F-s-k}, the right-hand side of \eqref{t-sum-hat-s} is contained in ${\rm R}(\mathcal F)$. Hence $t^{-\hat s^{(k)}_\mu}e_\mu$ is integral in ${\rm R}(\mathcal F)$, and we conclude that $t^{-\hat s^{(k)}_\mu}e_\mu\in{\rm R}(\mathcal F)$ since ${\rm R}(\mathcal F)$ is normal, which contradicts with \eqref{contra-1}. Hence \eqref{s-k-concave} is true.

Note that $s_\mu^{(1)}\in\Gamma(\mathcal F)$ for each $\mu\in \Delta_{L,1}$. Take $k=1$ in \eqref{s-k-G} and \eqref{s-k-concave}, we can define a piece-wise linear concave function $f: {\Delta_L}\to\mathbb R$ so that $\min f=0$ and
\begin{align}\label{s-1-G}
s_\mu^{(1)}=\max\{s\in\Gamma(\mathcal F)|s\leq f(\mu)\},~\forall \mu\in\Delta_{L,1}.
\end{align}
Indeed, consider the convex hull of
\begin{align}\label{graph-f}
\Delta_L':=\{(\lambda,s^{(1)}_\lambda)|\lambda\in\Delta_{L,1}\}.
\end{align}
It is a convex polytope in $ \Delta_L\times\mathbb R$ such that
$${\rm Conv}(P'_+)=\{(x,y)\in \Delta_L\times\mathbb R|x\in {\Delta_L},0\leq y\leq f(x)\}$$
for some concave function $f$. Obviously, $s^{(1)}_\mu\leq f(\mu)$ and the equality holds if $\mu$ is a vertex of a domain of linearity.
Combining with \eqref{s-k-concave}, we get \eqref{s-1-G}.
Also, by the normalization condition $\min\Gamma(\mathcal F,1)=0$ and \eqref{s-1-G},
there must be a vertex $\mu$ of $\Delta_L$ so that $s_\mu^{(1)}=0$. Hence $\min f=f(\mu)=0$.

It is easy to see that each domain of linearity of $f$ is a convex polytope with vertices in $\Delta_{L,1}$. In fact, any domains of linearity of $f$ is the projection of a facet of ${\rm Conv}(P'_+)$ on the roof $\{(x,f(x))|x\in {\Delta_L}\}$. But the vertices of such a facet must lie in $P'_+$. Hence it projects to a point in $\Delta_{L,1}$.

Note that $\mathcal V(G/H)$ is the antidominant Weyl chamber of the spherical roots $\Pi_{G/H}$. By \eqref{s-compare}, the gradient of $f$ (in the sense of subdifferential),
\begin{align}\label{nabla-f-1}
\nabla f\subset\mathcal V_\mathbb R(G/H).
\end{align}
Hence by using the $W$-action we can extend $f$ to a $W$-invariant piecewise-linear function, which is globally concave on $\Delta_L$.

\emph{Step-1.3: Proof of \eqref{s-k-G-final}.} It remains to prove that \eqref{s-k-G-final} holds for any $k\in\mathbb N_+$. 
Fix any $\mu\in \Delta_{L,k}$. By \eqref{s-k-G}, we can assume that there are $\lambda_1,...,\lambda_k\in \Delta_{L,1}$ so that
$$s_\mu^{(k)}=\sum_{i=1}^ks^{(1)}_{\lambda_i},~ V_\mu\subset V_{\lambda_1}\cdot...\cdot V_{\lambda_k}.$$
In particular,
\begin{align}\label{sum-lambda-mu}
\sum_{i=1}^k\lambda_i\succeq\mu.
\end{align}

We have
\begin{align*}
\frac1k s_\mu^{(k)}=\frac1k \sum_{i=1}^ks^{(1)}_{\lambda_i}\leq\frac1k\sum_{i=1}^kf({\lambda_i})\leq f(\frac{\sum_{i=1}^k\lambda_i}k).
\end{align*}
By \eqref{sum-lambda-mu} and \eqref{nabla-f-1} we get
\begin{align}\label{s-leq-f}
\frac1k s_\mu^{(k)}\leq f(\frac{\sum_{i=1}^k\lambda_i}k)\leq f(\frac{\mu}k).
\end{align}

Recall that the vertices of each domain of linearity of $f$ is in $\mathfrak M$. Suppose that $\Omega$ is the domain of linearity which contains $\frac\mu k$ whose vertices are $\lambda_1,...,\lambda_l\in\mathfrak M$, then there are non-negative constants $c_1,...,c_l$ such that 
\begin{align*}
\mu=&\sum_{i=1}^lc_ik\lambda_i,~\sum_{1=1}^lc_i=1.
\end{align*}
Since $f$ is linear on $\Omega$, we get
\begin{align}\label{mu-lambda-comb}
k\sum_{j=1}^lc_js^{(1)}_{\lambda_j}=kf(\frac \mu k).
\end{align}
On the other hand, by \eqref{s-k-concave}, we have
\begin{align*}
s^{(k)}_\mu&\geq\sup\{s\in\Gamma(\mathcal F)|s\leq\sum_{j=1}^lc_js^{(k)}_{k\lambda_j}\}\\
&\geq\sup\{s\in\Gamma(\mathcal F)|s\leq k\sum_{j=1}^lc_js^{(1)}_{\lambda_j}\}=\sup\{s\in\Gamma(\mathcal F)|s\leq kf(\frac\mu k)\},
\end{align*}
where we used \eqref{s-k-G} and \eqref{mu-lambda-comb} in the last inequality.

Combining with \eqref{s-leq-f}, up to replacing $f(\cdot)$ by $\frac1{r_0}f(r_0\cdot)$ we get
\begin{align*}
s_{\lambda}^{(k)}=\sup\{s\in\Gamma(\mathcal F)|s\leq kf(\frac\lambda k)\},~\forall \lambda\in \Delta_{L,k}~\text{and}~k\in\mathbb N.
\end{align*}
The relation \eqref{s-k-G-final} then follows from the above equality and the fact that $s^{(k)}_\lambda\in\Gamma(\mathcal F)$. Also, from \emph{Step-1.2} we see that after this scaling, $f$ is still a concave piecewise linear function on $\Delta_L$ with $\nabla f\in\mathcal V_\mathbb R(G/H)$. But its domains of linearity consists of convex polytopes with vertices in $\mathfrak M_\mathbb Q$.


\textbf{Part-2: The inverse direction.}


\emph{Step-2.1: Construction of the $\mathbb R$-test configuration.} Given an $f$ satisfying the assumption of Theorem \ref{G-classify}, we can fix an $r_0\in\mathbb N_+$ so that the domains of linearity of the function
$$\hat f(x)=r_0f(\frac 1{r_0}x): {r_0\Delta_L}\to\mathbb R$$
consist of polytopes with vertices in $\mathfrak M+\chi_0$. Replacing $L$ by $L^{r_0}$, we may assume that $r_0=1$. By concavity and \eqref{nabla-f-1}, it holds
$$(k_1+k_2)f(\frac\mu{k_1+k_2})\geq(k_1+k_2)f(\frac{\lambda_1+\lambda_2}{k_1+k_2})\geq k_1f(\frac{\lambda_1}{k_1})+k_2f(\frac{\lambda_2}{k_2})$$
for any two $\lambda_i\in \Delta_{L,k_i},\, i=1,2$ and $\mu\in \Delta_{L,k_1+k_2}$ so that $\lambda_1+\lambda_2\succeq\mu$.
Combining with \cite[Proposition 4]{Timashev-Sbo}, it is direct to check that \eqref{F-s-k} and \eqref{s-k-G-general-inv} define a $G$-equivariant filtration $\mathcal F$ of $(M,L)$. Then we prove that $\mathcal F$ is an $\mathbb R$-test configuration. We have two case:

\emph{Case-2.1.1.} $\Gamma$ is a discrete subgroup in $\mathbb R$. In this case \eqref{s-k-G-general-inv} is reduced to
\begin{align*}
s_{\lambda}^{(k)}=\max\{s\in\Gamma|s\leq kf(\frac\lambda k)\},~\forall \lambda\in \Delta_{L,k}~\text{and}~k\in\mathbb N.
\end{align*}
Clearly $\Gamma(\mathcal F)\subset\Gamma$ is finitely generated and $\mathcal F$ is an $\mathbb R$-test configuration.

\emph{Case-2.1.2.} $\Gamma$ is a not discrete. Then $\Gamma$ is everywhere dense in $\mathbb R$. In this case \eqref{s-k-G-general-inv} is reduced to
\begin{align*}
s_{\lambda}^{(k)}=kf(\frac\lambda k),~\forall \lambda\in \Delta_{L,k}~\text{and}~k\in\mathbb N.
\end{align*}
Consequently $\Gamma(\mathcal F)$ is generated by a finite set $\{f(\lambda)|\lambda\in \Delta_{L,1}\}$. Again $\mathcal F$ is an $\mathbb R$-test configuration.

\emph{Step-2.2: Normality of the total space.} It remains to prove that the corresponding Rees algebra
$${\rm R}(\mathcal F)=\oplus_{k\in\mathbb N}\oplus_{s\in\Gamma(\mathcal F)}\oplus_{\lambda\in \Delta_{L,k},kf(\frac1k\lambda)\geq s}t^{-s} V_\lambda.$$
is normal. Since
$$({\rm R}(\mathcal F)\subset){\rm R}'=\oplus_{k\in\mathbb N}\oplus_{s\in\Gamma(\mathcal F)}\oplus_{\lambda\in \Delta_{L,k}}t^{-s} V_\lambda$$
is a normal ring, it suffices to show that ${\rm R}(\mathcal F)$ is integrally closed in ${\rm R}'$. This is equivalent to show that any $\bar\sigma\in{\rm R}'\setminus{\rm R}(\mathcal F)$ is not integral in ${\rm R}(\mathcal F)$.

Assume there is some $\bar\sigma\in{\rm R}'\setminus{\rm R}(\mathcal F)$ integral over ${\rm R}(\mathcal F)$. It suffices to deal with $\bar\sigma$ of the following form:
\begin{align}\label{bar-sigma-decop}
\bar\sigma=\sum_{i=1}^{d}t^{-s_i}\sigma_{\tau_i},
\end{align}
where each $\tau_j\in \Delta_{L,k(j)}$ for some $k(j)\in\mathbb N$,
\begin{align*}
(0\not=)\sigma_{\tau_j}\in V_{\tau_j}~\text{and}~ s_j >k(j)f(\frac1{k(j)}\tau_j),~j=1,...,d.
\end{align*}
Since $\bar\sigma$ is integral over ${\rm R}(\mathcal F)$, there is some $q\in\mathbb N_+$ such that for any integer $l\in\mathbb N_+$,
\begin{align}\label{normal-condition}
\bar\sigma^l\in{\rm R}(\mathcal F)+{\rm R}(\mathcal F)\bar\sigma^1+...+{\rm R}(\mathcal F)\bar\sigma^{q}.
\end{align}

Consider the convex hull of $\{w(\tau_i)|w\in W^{G/H},~i=1,...,d\}$. Choose a $\tau\in\{\tau_1,\dots, \tau_d\}$ which is a vertex of this convex hull. Then for any $p\in\mathbb N_+$, $\sum_{s=1}^p\tau_{i_s}\not\succeq p\tau$ whenever there is a $\tau_{i_s}\in\{\tau_1,\dots, \tau_d\}\setminus\{\tau\}$. Take the component $t^{-s}\sigma_{\tau}$ of $\bar{\sigma}$ in its decomposition \eqref{bar-sigma-decop}. 
Suppose that the degree of $t^{-s}\sigma_{\tau}$ is $m$ (that is, $\sigma_\tau\in R_m$). Then
\begin{align}\label{s-lower}
m(f(\frac1m\tau))<s.
\end{align}
For any $l\in\mathbb N_+$, we can decompose $\bar\sigma^l$ as \eqref{bar-sigma-decop}. By Lemma \ref{alg-lem} below, there is a nonzero component of $\bar\sigma^l$ in $t^{-ls} (V_{l\tau})$ with degree $lm$, we denote it by $t^{-ls}\sigma_{l\tau}$.

Decompose $\bar{\sigma}^i, i=0,\dots,q$ as \eqref{bar-sigma-decop}. Then all their components have the form $t^{-w}\sigma_{\gamma}$ for some degree $k$. All such triples $(w,\gamma, k)$ form a finite set $\mathcal S$. Thus
$$\bar\sigma^l\in\sum_{(w,\gamma,k)\in\mathcal S}{\rm R}(\mathcal F)t^{-w}\sigma_{\gamma}.$$

Consider the component $t^{-ls}\sigma_{l\tau}$ of $\bar\sigma^l$. Since $\mathcal S$ is finite, up to passing to a subsequence, there is some $(w,\gamma, k)\in\mathcal S$ such that
$$t^{-ls}\sigma_{l\tau}\in (t^{-w} V_{\gamma}){\rm R}(\mathcal F)_{lm-k},~l\in\mathbb N_+.$$
Thus, we have
$$t^{-ls}\sigma_{l\tau}\in (t^{-w} V_{\gamma})(t^{-r} V_{\mu})$$ for some $t^{-r} V_{\mu}\subseteq {\rm R}(\mathcal F_f)_{lm-k}$, where $ls=w+r,$
\begin{align}\label{gamma+mu}
\gamma+\mu\succeq l\tau,
\end{align}
and
\begin{align*}
r\leq (lm-k)f(\frac{\mu}{lm-k}).
\end{align*}
Here in \eqref{gamma+mu} we used \cite[Proposition 4]{Timashev-Sbo}. Hence
\begin{align*}
s=\frac wl+\frac rl\leq& \frac wl+(m-\frac kl)f(\frac1{lm-k}{\mu})\\
\leq&\frac wl+(m-\frac kl)f(\frac{\tau-\frac1l{\gamma}}{m-\frac kl}),~\text{for sufficiently large}~l\in\mathbb N_+,
\end{align*}
where in the last line we used \eqref{gamma+mu} and \eqref{nabla-f-1}. Sending $l\to+\infty$  we see that $$s\leq mf(\frac1m\tau),$$ which contradicts to \eqref{s-lower}. Hence \eqref{normal-condition} is not true and we conclude that ${\rm R}(\mathcal F)$ is normal.

\end{proof}

To complete the arguments in \emph{Step-2.2}, we need the following

\begin{lem}\label{alg-lem}
Suppose that $\sigma\in V_\mu\subset R_k$
for some $k\in\mathbb N$. Then for any $q\in\mathbb N_+$, $\sigma^q$ has non-zero component in $V_{q\mu}$.
\end{lem}

\begin{proof}
Since $V_\mu$ is an irreducible $G$-representation, there is a $\hat g_0\in G$ so that $\hat g_0(\sigma)$ has non-zero component on the subspace generated by the highest weight vector. Thus for any $q\in\mathbb N_+$, $\hat g_0(\sigma^q)$ has non-zero component in $V_{q\mu}$. We conclude the Lemma since
$\sigma^q=\hat g_0^{-1}(\hat g_0(\sigma^q))$ and the fact that $V_{q\mu}$ is $G$-invariant.
\end{proof}

\begin{rem}\label{f-Gamma-resclae}
In the second part of Theorem \ref{G-classify}, if we replace the data $(f,\Gamma)$ by $(cf,c\Gamma)$ with $c>0$, then the total space and the central fibre of the resulting $\mathbb R$-test configuration leave unchanged. However, the induced vector field on the central fibre is rescaled by $c$.
\end{rem}

\begin{rem}\label{ZTC}
When $f$ is rational, we can take $k$ to be the smallest positive integer so that the set
$$\{k(\lambda,s)|0\leq s\leq f(\lambda),~\lambda\in {\Delta_L}\}$$
is an integral polytope in $(\mathfrak M(G/H)+\chi_0)\times\mathbb Z$ and $\Gamma=\mathbb Z$. Theorem \ref{G-classify} then reduces to the classification theorem of $G$-equivariant normal $\mathbb Z$-test configurations (cf. \cite[Theorem 4.1]{Del-2020-09}).
\end{rem}

\subsection{Equivariant $\mathbb R$-test configurations with reduced central fibre}

We can further classify $G$-equivariant normal $\mathbb R$-test configurations with reduced central fibre by using Theorem \ref{G-classify}.

\begin{thm}\label{G-classify-reduced}
Let $(X,L)$ be a polarized spherical embedding of $G/H$ with moment polytope $\Delta_L$. Then for any $G$-equivariant normal $\mathbb R$-test configuration $\mathcal F$ of $(M,L)$ with reduced central fibre, there is a concave, piecewise-linear function $f\geq\min f=0$ with $\nabla f\in\mathcal V_\mathbb R(G/H)$, and the domains of linearity of $f$ consist of rational polytopes in $\mathfrak M_\mathbb Q$, such that
\begin{align}\label{s-k-G-final-reduced}
s_{\lambda}^{(k)}=kf(\frac\lambda k),~\forall \lambda\in \Delta_{L,k}~\text{and}~k\in\mathbb N,
\end{align}
and vice versa.
\end{thm}

In the following we will call $f$ in \eqref{s-k-G-final-reduced} the \emph{function associated to $\mathcal F$} and denote $\mathcal F=\mathcal F_f$.

\begin{proof}[Proof of Theorem \ref{G-classify-reduced}]
We divide the proof in two parts.

\textbf{Part-1: Necessity of \eqref{s-k-G-final-reduced}.}
Suppose that $\mathcal F$ is given so that ${\rm Gr}(\mathcal F)$ defined by \eqref{GrF-def} contains no nilpotent element. Let $f$ be the function defined in Theorem \ref{G-classify}.
We will show \eqref{s-k-G-final-reduced} holds.
Otherwise, by \eqref{s-k-G-final} there is a $\lambda_0\in \Delta_{L,l_0}$ for some $l_0\in\mathbb N_+$ so that
$$(\Gamma(\mathcal F)\ni)s_{\lambda_0}^{(l_0)}<l_0f(\frac{\lambda_0}{l_0}).$$
Let $\sigma_0\in V_{\lambda_0}$ be a highest weight vector. Then $\sigma_0^{\otimes k}\in V_{k\lambda_0}$ has (real) weight $t^{-ks_{\lambda_0}^{(l_0)}}$ for any $k\in\mathbb N_+$.

On the other hand, choose a set of positive generators $\{s_1,...,s_{r_{\mathcal F}}\}$ of ${\mathcal F}$ and $s_M:=\max\{s_1,...,s_{r_{\mathcal F}}\}$. Then there is a $k_0\in\mathbb N_+$ so that
$$k_0(l_0f(\frac{\lambda_0}{l_0})-s_{\lambda_0}^{(l_0)})\geq s_M.$$
Hence there must be an $s'\in\Gamma(\mathcal F)$ so that $k_0s_{\lambda_0}^{(l_0)}<s<k_0l_0f(\frac{\lambda_0}{l_0})$ and consequently
$$s_{k_0\lambda_0}^{(k_0l_0)}>k_0s_{\lambda_0}^{(l_0)}.$$
Hence $V_{k_0\lambda_0}\subset \mathcal F^{>k_0s_{\lambda_0}^{(l_0)}}R_{k_0l_0}$ and in \eqref{GrF-def} the piece
$$\mathcal F^{k_0s_{\lambda_0}^{(l_0)}}R_{k_0l_0}/\mathcal F^{>k_0s_{\lambda_0}^{(l_0)}}R_{k_0l_0}$$
contains no $ (V_{k_0\lambda_0})$-factor. Hence $\sigma_0^{\otimes k}$ descends to $0$ in ${\rm Gr}(\mathcal F)$. In other words, $\sigma_0^{\cdot k}=0$ in ${\rm Gr}(\mathcal F)$ and $\sigma_0$ is nilpotent, a contrdiction.

\textbf{Part-2: Sufficiency of \eqref{s-k-G-final-reduced}.} Suppose that \eqref{s-k-G-final-reduced} holds. Define
$$\mathcal F^sR_k:=\oplus_{\lambda\in \Delta_{L,k},kf(\lambda/k)\geq s} V_\lambda,~\forall k\in\mathbb N_+.$$
We will show \eqref{GrF-def} contains no nilpotent element. Otherwise, there are a $\sigma\in V_{\lambda_0}$ for some $\lambda_0\in  \Delta_{L,l_0}$ and some $k_0\in\mathbb N_+$ so that $\sigma_0^{\cdot k}=0$ in ${\rm Gr}(\mathcal F)$.

By Lemma \ref{alg-lem}, we can assume that $\sigma$ has non-zero component on the direction of highest weight vector. Hence $\sigma_0^{\cdot k}$ has non-zero component in  $\sigma\in V_{k_0\lambda_0}$. It must hold
$$\mathcal F^{k_0s_{\lambda_0}^{(l_0)}}R_{k_0l_0}/\mathcal F^{>k_0s_{\lambda_0}^{(l_0)}}R_{k_0l_0}$$
contains no $ (V_{k_0\lambda_0})$-factor. This implies
$$s_{k_0\lambda_0}^{(k_0l_0)}>k_0s_{\lambda_0}^{(l_0)},$$
a contradiction to \eqref{s-k-G-final-reduced}.
\end{proof}

\begin{rem}\label{RTC-reduced}
Given any $f$ satisfying the assumption of Theorem \ref{G-classify-reduced}, we can construct $\mathcal F_f$ by choosing the following data in Theorem \ref{G-classify}:
\begin{itemize}
\item $k$ is the smallest positive integer so that the domains of linearity of $f$ in $\Delta_L$ consists of polytopes with vertexes in $\mathfrak M(G/H)+\chi_0$.
\item $\Gamma$ is the group generated by $\{s_{\lambda}^{(k)}|s_{\lambda}^{(k)}~\text{is given by}~\eqref{s-k-G-final-reduced}\}$.
\end{itemize}
It is direct to check that $\Gamma(\mathcal F_f)=\Gamma$.
\end{rem}

Now we are going to classify $G$-equivariant special $\mathbb R$-test configurations. We have the following variant of Proposition \ref{RTC-appr} for spherical varieties,

\begin{prop}\label{F-Lambda-appro}
Let $(X,L)$ be a polarized $G$-spherical variety with moment polytope $\Delta_L$. Suppose that $\mathcal F_\Lambda$ is the $G$-equivariant normal $\mathbb R$-test configuration of $(X,L)$ associated to the function \eqref{f-affine}. Then there is a rational $\Lambda'\in\mathcal V(G/H)$ so that $\mathcal F_{\Lambda'}$ has the same central fibre with $\mathcal F_\Lambda$.
\end{prop}

\begin{proof}
As in the proof of Proposition \ref{RTC-appr}, suppose that $R(X,L)$ is generated by $R_1$ and $\mathcal F_\Lambda$ by the piece $\mathcal F_\Lambda R_{1}$. Recall that the integral points in the 1-st piece of the Okunkov body
$$\cup_{\lambda\in\Delta_{L,1}}(\lambda,\frac1k\Delta^{\rm GT}(k \lambda))$$
are in one-one correspondence with a set of sections in $R_1$ that also forms a basis. Let $\mathbf T$ be the torus induced by $\mathcal F_\Lambda$ on $R_1=\oplus_{\lambda\in\Delta_{L,1}}V_\lambda$. Then the $\mathbf T$-weight of an integral point (hence a section in $R_1$) in $(\lambda,\frac1k\Delta^{\rm GT}(k \lambda))$ is $\Lambda(\lambda)$.

Let $\mathfrak Y_1:=\Delta_{L,1}$ and $\mathfrak Y_k$ be the sum of any $k$ elements (possibly with repeated terms) of $\mathfrak Y_1$. From the proof of Proposition \ref{RTC-appr}, it suffices to solve the system the system \eqref{eq-Lambda-1}-\eqref{eq-Lambda-2} on the rational cone $\mathcal V_\mathbb R(G/H)$. Again, the above system of $\Lambda'$ defines a rational cone $\mathfrak Z\subset\mathcal V_\mathbb R(G/H)$ that contains $\Lambda$. Choose any rational $\Lambda'\in\mathfrak Z$ we get the Proposition.
\end{proof}

A geometric proof of Proposition \ref{F-Lambda-appro} has been given in \cite{Li-Wang}.

\begin{prop}\label{R-TC-Lambda}
Let $(X,L)$ be a polarized spherical embedding of $G/H$.
Then for any $\Lambda\in \mathcal V_\mathbb R(G/H)$ there is a $G$-equivariant special $\mathbb R$-test configuration $\mathcal F_{\Lambda}$ of $(X,L)$,
so that $\mathcal X_0$ is a $G$-spherical variety and admits an action of the torus $\overline{\exp(t\Lambda)}\subset G$.
Conversely, for any $G$-equivariant special $\mathbb R$-test configuration $\mathcal F$ of $(X,L)$, there is a $\Lambda\in \mathcal V_\mathbb R(G/H)$ such that $\mathcal F=\mathcal F_{\Lambda}$.
\end{prop}
\begin{proof}

We first show that a $G$-equivariant normal $\mathbb R$-test configuration $\mathcal F$ has irreducible central fibre if and only if the associated function $f$ given by Theorem \ref{G-classify-reduced} is affine on ${\Delta_L}$.

Assume that $\mathcal X_0$ is irreducible. Then ${\rm Gr}(\mathcal F)$ is integral.
We show that $f$ is affine.
Otherwise, we can take two domains of linearity $Q_1,Q_2\subset\Delta_L$ so that they intersect along a common facet.
Take $\lambda_i\in Q_i\cap(\mathfrak M_\mathbb Q+\chi_0)$ so that the line segment $\overline{\lambda_1\lambda_2}\subset Q_1\cup Q_2$.
Up to replacing $\Delta_L$ by some $k_0\Delta_L$, we can assume $\lambda_i\in Q_i\cap(\mathfrak M+k_0\chi_0)$ for $i=1,2$.

Let $\sigma_i\in V_{\lambda_i}$ be a highest weight vector. Then $\sigma_i$ has real weight $t^{-s_{\lambda_i}^{(1)}}$. Consequently, $\sigma_1\otimes \sigma_2\in V_{\lambda_1+\lambda_2}$ has real weight $t^{-s_{\lambda_1}^{(1)}-s_{\lambda_2}^{(1)}}$. On the other hand, for $\lambda_1+\lambda_2\in\overline{2\Delta_L}\cap(\mathfrak M+2\chi_0)=\Delta_{L,2}$, by \eqref{s-k-G-final-reduced},
$$s_{\lambda_1+\lambda_2}^{(2)}=2f(\frac12(\lambda_1+\lambda_2))>s_{\lambda_1}^{(1)}+s_{\lambda_1}^{(2)},$$
where the last inequality follows from the concavity of $f$ and the fact that $\lambda_i$'s lie in different domains of linearity. Hence $\sigma_1\cdot\sigma_2=0$ in ${\rm Gr}(\mathcal F)$. A contradiction to the assumption that ${\rm Gr}(\mathcal F)$ is integral.

Conversely, assume that $f$ is affine. By Theorem \ref{G-classify-reduced}
\begin{align}\label{f-affine}
f_\Lambda=\Lambda(\lambda)-\min_{\mu\in\Delta_L}\Lambda(\mu),
\end{align}
for some $\Lambda\in\mathcal V_\mathbb R(G/H)$. We will show that ${\rm Gr}(\mathcal F)$ is integral.
Otherwise, there are $(0\not=)\sigma_i\in V_{\lambda_i},i=1,2$ so that $\sigma_1\cdot\sigma_2=0$ in ${\rm Gr}(\mathcal F)$.
By Lemma \ref{alg-lem}, we can assume that each $\sigma_i$ is a highest weight vector.
Assume that $\lambda_i\in\overline{k_i\Delta_L}$. Then by \eqref{s-k-G-final-reduced}, $\sigma_i$ has real weight $t^{-k_if(\lambda_i/k_i)}$. Consequently, $\sigma_1\otimes \sigma_2\in V_{\lambda_1+\lambda_2}$ has real weight $t^{-k_1f(\lambda_1/k_1)-k_2f(\lambda_2/k_2)}$ with
$$k_1f(\frac{\lambda_1}{k_1})+k_2f(\frac{\lambda_2}{k_2})=(k_1+k_2)f(\frac{\lambda_1+\lambda_2}{k_2+k_2}).$$
Here we used the fact that $f$ is affine. Note that the right-hand side is just the real weight of the $V_{\lambda_1+\lambda_2}$-piece in ${\rm Gr}(\mathcal F)$. Hence $\sigma_1\cdot\sigma_2\not=0$, a contradiction.

We have seen that $\mathcal X_0$ is irreducible if and only if $f$ is affine. It remains to show that the central fibre of an arbitrary $\mathcal F_\Lambda$ is normal. By Proposition \ref{F-Lambda-appro}, the central fibre of $\mathcal F_\Lambda$ can be realized as that of a $G$-equivariant normal $\mathbb Z$-test configuration $\mathcal F_{\Lambda'}$. On the other hand, by \cite[Theorem 15.20]{Timashev-book}, each irreducible component of the central fibre of a $G$-equivariant normal $\mathbb Z$-test configuration is normal. Hence the central fibre $\mathcal X_0$ of $\mathcal F_{\Lambda'}$ is normal.

Note that by \eqref{GrF-def} and \eqref{s-k-G-final-reduced}, the vector field induced by $\mathcal F_\Lambda$ is precisely $-\Lambda$. Hence we get the Proposition.

\end{proof}
\begin{rem}\label{R-val}
From the proof of \cite[Lemma 2.1 and Theorem 3.3]{Timashev-survey}, we see that $G$-invariant $\mathbb R$-valuations are mapped bijectively to points in $\mathcal V_\mathbb R(G/H)$. Moreover, it is also direct to see that up to a translation, the filtration induced by the $\mathbb R$-valuation $\Lambda$ is precisely $\mathcal F_\Lambda$.

In fact, each $\Lambda\in\mathcal V_\mathbb R(G/H)$ induces a quasimonomial valuation. Using \cite[Section 2.6]{Fulton}, for each $G$-invariant $\mathbb R$-valuation $\Lambda\in\mathcal V_\mathbb R(G/H)$, we can always construct a set $\{v_1,...,v_r\}\subset\mathcal V(G/H)$ which is a $\mathbb Z$-basis of $\mathfrak N(G/H)$, such that $\Lambda$ lies in the cone $\mathcal C$ generated by $\{v_1,...,v_r\}$. In fact, we divided $\mathcal V_\mathbb R(G/H)$ into several simplicial subcones generated by a $\mathbb Z$-basis of $\mathfrak N(G/H)$ and take one of those cones containing $\Lambda$.
By \cite[Proposition 3.3]{Pasquier-survey}, the colourless cone $(\mathcal C,\emptyset)$ defines a smooth, toroidal model $\tilde X$, and it is direct to check that $\Lambda$ is a monomial valuation on $\tilde X$. Hence $\Lambda$ is a quasimonomial valuation on $G/H$.
\end{rem}

\section{Central fibre of a $G$-equivariant special $\mathbb R$-test configuration of $\mathbb Q$-Fano spherical varieties}
In this section, we consider the central fibre of a $G$-equivariant special $\mathbb R$-test configuration of a polarized spherical variety constructed in Proposition \ref{R-TC-Lambda}. 
Let $f(y)=\Lambda(y)-\min_{y\in\Delta_L}\Lambda(y)$ be an affine function on $\Delta_L$ with gradient $\Lambda\in\mathcal V(G/H)$ and $\mathcal F$ the special ${G}$-equivariant $\mathbb{R}$-test configuration $\mathcal{F}$ of $(X, K_X^{-1})$ with points of discontinuity
\begin{align}\label{point-dis}
s_{\lambda}^{(k)}=\Lambda(\lambda)-k\min_{y\in\Delta_L}\Lambda(y), \forall \lambda \in \Delta_{L,k} \text { and } k \in \mathbb{N}.
\end{align}
It is also know that the central fibre $\mathcal X_0$ is a polarized projective spherical embedding of some spherical homogeneous space $G/H_0$ with moment polytope $\Delta_L$ (cf. \cite[Proposition 4]{Popov-1986} or using the approximating sequence constructed by Proposition \ref{RTC-appr}). In particular, in the case when $X$ is $\mathbb Q$-Fano and $L=K_X^{-1}$, $\mathcal X_0$ can be recovered from the above data.

\subsection{Homogenous spherical datum of $G/H_0$}
We will first determine $G/H_0$. By \cite[Theorem 30.22]{Timashev-book} (based on \cite{Losev,CF-2}, see also \cite{BraP1,BraP2, Avdeev-CF}), it suffices to determine its homogeneous spherical datum. Denote by $\varrho_0(\cdot)$ the corresponding map \eqref{varrho-def} of $G/H_0$. We have 
\begin{prop}\label{homo-sph-datum-G/H0}
Let $(X,L)$ be a polarized $G$-spherical variety with moment polytope $\Delta_L$. Suppose that $\mathcal F_\Lambda$ is a $G$-equivariant special $\mathbb R$-test configuration of $(X,L)$. Then the central fibre $\mathcal X_0$ is a spherical embedding of $G/H_0$ with moment polytope $\Delta=\Delta_L$.
The homogeneous spherical datum $(\mathfrak M(G/H_0)$, $\Pi_{G/H_0}$, $\Pi_{G/H_0}^p$, $\mathcal D^a(G/H_0))$ of $G/H_0$ consists of
\begin{align*}
\mathfrak M(G/H_0)=\mathfrak M(G/H),\,
\Pi_{G/H_0}=\Pi_{G/H}\cap\Lambda^\perp,\,
\Pi_{G/H_0}^p=\Pi_{G/H}^p,
\end{align*}
and for any $D_0\in\mathcal D(G/H_0;\alpha)$ with $\alpha\in\Pi_{G/H_0}(\subset\Pi_{G/H})$, there is a unique $D\in\mathcal D^a(G/H)\cap\mathcal D(G/H,\alpha)$ so that $\varrho_0(D_0)=\varrho(D)$.
\end{prop}

Clearly, there are only finitely many choices of the datum $(\mathfrak M(G/H_0)$, $\Pi_{G/H_0}$, $\Pi_{G/H_0}^p$, $\mathcal D^a(G/H_0))$.
By the deep uniqueness theorem of \cite{Losev} (see also \cite[Theorem 30.22]{Timashev-book}), there are only finitely many choices of $G/H_0$.

\begin{coro}\label{G/H0-finite}
If $\Lambda_i,~i=1,2$ are two vectors that lie in the relative interior of a same face of $\mathcal V_\mathbb R(G/H)$, then the central fibres $X_{0,i}$ of $\mathcal F_{\Lambda_i},~i=1,2$ are projective completion of a same $G$-spherical homogenous space.
\end{coro}

In the following we prove Proposition \ref{homo-sph-datum-G/H0}. We first compute the lattice of $B$-semiinvariant functions $\mathfrak M(G/H_0)$,

\begin{prop}\label{M(G/H0)}
The lattice $\mathfrak M(G/H_0)=\mathfrak M(G/H).$
\end{prop}
\begin{proof}
From \eqref{GrF-def} we see that for $\mathcal X_0$, the moment polytope $\Delta_{\mathcal L_0,k}(\mathcal X_0)=\Delta_{L,k}(X)$ for all $k\in\mathbb N$. In particular, if there is a section $s_0\in{\rm H}^0(X,L)^{(B)}_{\chi_0}$ for some $\chi_0\in\mathfrak X(B)$, then there also exists a  $\bar s_0\in{\rm H}^0(\mathcal X_0,\mathcal L_0)^{(B)}_{\chi_0}$. On the other hand, by \eqref{Delta-L-k}, for sufficiently large $k$, the polytope of the divisor $k{\rm div}(s_0)$, $$\Delta_{k{\rm div}(s_0)}=\Delta_{\mathcal L_0,k}(\mathcal X_0)-k\chi_0$$
is a polytope in $\mathfrak M(G/H)$ that contains a basis of $\mathfrak M(G/H)$. Similarly, $$\Delta_{k{\rm div}(\bar s_0)}=\Delta_{\mathcal L,k}(\mathcal X)-k\chi_0$$ also contains a basis of $\mathfrak M(G/H_0)$. Hence
\begin{align*}
\mathfrak M(G/H_0)=&{\rm Span}_{\mathbb Z}\{\Delta_{\mathcal L_0,k}(\mathcal X_0)-k\chi_0\}\\
=&{\rm Span}_{\mathbb Z}\{\Delta_{L,k}(X)-k\chi_0\}=\mathfrak M(G/H).
\end{align*}

\end{proof}

Then we compute the spherical roots of $G/H_0$.
\begin{prop}\label{val-cone-G/H0-prop}
Let $\Pi_{G/H}:=\{\alpha_1,...,\alpha_r\}$ be the spherical roots of $G/H$. Suppose that $\Lambda$ satisfies
\begin{align}
\alpha_i(\Lambda)=&0,~i=1,...,i_0;\label{zero-roots}\\
\alpha_i(\Lambda)<&0,~i=i_0+1,...,r.\label{non-zero-roots}
\end{align}
Then the valuation cone of $G/H_0$ is
\begin{align}\label{val-cone-G/H0}
\mathcal V(G/H_0)=\{x\in\mathfrak N_\mathbb R(G/H)|\alpha_i(x)\leq0,~i=1,...,i_0\}.
\end{align}
Consequently, the spherical roots of $G/H_0$ is
\begin{align}\label{sip-sph-rt-G/H0}
\Pi_{G/H_0}=\{\alpha_i|i=1,...,i_0\}.
\end{align}
\end{prop}

\begin{proof}
Suppose that $\lambda_j\in \Delta_{L,k_j}$ for $j=1,2$. Then in the Kodaira ring of $X$,
\begin{align}\label{mult-rule}
V_{\lambda_1}\cdot V_{\lambda_1}=V_{\lambda_1+\lambda_2}\oplus(\oplus_{\beta} V_{\lambda_1+\lambda_2-\beta}),
\end{align}
where each $\beta\not=0$ is a non-negative combination of simple roots in $\Phi^G_+$. Fix any $V_{\lambda_1+\lambda_2-\beta_0}$ appeared in \eqref{mult-rule} so that $\Lambda(\beta_0)<0$. Then
\begin{align}\label{f-lambda1+lambda2-beta+}
f(\frac{\lambda_1+\lambda_2-\beta_0}{k_1+k_2})>f(\frac{\lambda_1+\lambda_2}{k_1+k_2})=\frac{k_1f(\frac{\lambda_1}{k_1})+k_2f(\frac{\lambda_2}{k_2})}{k_1+k_2}
\end{align}
By \eqref{point-dis},
\begin{align}\label{s-l1+l2-beta}
s^{(k_1+k_2)}_{\lambda_1+\lambda_2-\beta_0}>s^{(k_1+k_2)}_{\lambda_1+\lambda_2}=s^{(k_1)}_{\lambda_1}+s^{(k_2)}_{\lambda_2}.
\end{align}
Similarly, if $\beta'_0(\Lambda)=0$, then
\begin{align}\label{s-l1+l2-beta}
s^{(k_1+k_2)}_{\lambda_1+\lambda_2-\beta'_0}=s^{(k_1+k_2)}_{\lambda_1+\lambda_2}=s^{(k_1)}_{\lambda_1}+s^{(k_2)}_{\lambda_2}.
\end{align}

Consider the two factors $t^{-s^{(k_j)}_{\lambda_j}}V_{\lambda_j},~j=1,2$ in ${\rm Gr}(\mathcal F)$ \eqref{GrF-def}. By \eqref{mult-rule}, the multiplication of the two blocks in ${\rm Gr}(\mathcal F)$ is
\begin{align}\label{mult-rule-X0}
t^{-s^{(k_1)}_{\lambda_1}}V_{\lambda_1}\cdot t^{-s^{(k_2)}_{\lambda_2}}V_{\lambda_2}=t^{-s^{(k_1)}_{\lambda_1}-s^{(k_2)}_{\lambda_2}}V_{\lambda_1+\lambda_2}\oplus\oplus_\beta(...),
\end{align}
where $\beta$'s appeared in \eqref{mult-rule-X0} form a subset of those appeared in \eqref{mult-rule}.

But by \eqref{s-l1+l2-beta},
$$V_{\lambda_1+\lambda_2-\beta_0}\subset\mathcal F^{>s^{(k_1+k_2)}_{\lambda_1+\lambda_2}}R_{k_1+k_2}$$
whenever $\Lambda(\beta_0)<0$. Hence the $\beta$'s appeared in \eqref{mult-rule-X0} if and only if $\Lambda(\beta)=0$. Consequently, there is no $c_i\alpha_i$ with $c_i>0$ and $i>i_0$ lies. We conclude
$$\mathcal V(G/H_0)=-({\rm Span}_{\mathbb R_{\geq0}}\Pi_{G/H_0})^\vee=-({\rm Span}_{\mathbb R_{\geq0}}\{\alpha_1,...,\alpha_{i_0}\})^\vee,$$
which is precisely \eqref{val-cone-G/H0}. Taking primitive generators of each edge of $-(\mathcal V(G/H_0))^\vee$ one gets \eqref{sip-sph-rt-G/H0}.
\end{proof}


From Proposition \ref{val-cone-G/H0-prop} we directly get
\begin{coro}\label{Lambda-appro}
Let $\mathcal F_\Lambda$ be the $G$-equivariant special $\mathbb R$-test configuration as above and $\mathcal F_{\Lambda'}$ be the $G$-equivariant special $\mathbb Z$-test configuration constructed in Proposition \ref{F-Lambda-appro} whose central fibre is isomorphic to that of $\mathcal F_\Lambda$. Then $\Lambda$ and $\Lambda'$ are contained in the relative interior of a same face of $\mathcal V_{\mathbb R}(G/H)$.
\end{coro}
\begin{proof}
Clearly, the sets spherical roots (which is a subset of $\Pi_{G/H}$) of the two central fibres coincide. By Proposition \ref{val-cone-G/H0-prop}, a spherical root in $\Pi_{G/H}$ is orthogonal to $\Lambda'$ if and only if it is orthogonal to $\Lambda$. We conclude the Corollary since $\mathcal V_{\mathbb R}(G/H)$ is the dual cone of the cone spanned by $\Pi_{G/H}$.
\end{proof}
Also, we have
\begin{coro}\label{horo-X0}
$\mathcal F_\Lambda$ has horospherical central fibre $\mathcal X_0$ if and only if $\Lambda\in{\rm RelInt}(\mathcal V_\mathbb R(G/H))$.
\end{coro}

\begin{proof}
By \cite[Propositon 7.6]{Timashev-book}, $\mathcal X_0$ is horospherical if and only if in the graded algebra \eqref{GrF-def} it holds
$$V_\lambda\cdot V_\mu\subset V_{\lambda+\mu},$$
for any $\lambda,\mu$. This is equivalent to $\Pi_{G/H_0}=\emptyset$ and the Corollary follows from \eqref{sip-sph-rt-G/H0}.
\end{proof}
We remark that Corollary \ref{horo-X0} was first shown in \cite{Popov-1986} for a wider class of $G$-varieties.

Now we compute the remaining terms $\Pi^p_{G/H_0}$ and $\mathcal D^a(G/H_0)$. We will reduce the arguments to $G$-equivariant normal $\mathbb Z$-test configurations.

Suppose that $\Lambda$ is contained in the relative interior of a face $\mathfrak F$ of $\mathcal V_\mathbb R(G/H)$. Let $\Lambda'$ be a rational vector so that $\mathcal F_{\Lambda'}$ is a $G$-equivariant special test configuration whose central fibre is isomorphic to that of $\mathcal F_\Lambda$. From Corollary \ref{Lambda-appro} we have $\Lambda'\in{\rm RelInt}(\mathfrak F)$.

Let us recall the geometric structure of a general $G$-equivariant normal $\mathbb Z$-test configuration. Let $\Lambda'$ be any rational element in $\mathcal V(G/H)$. Recall Remark \ref{f-Gamma-resclae}. We scale  $\Lambda'$ by a positive $m'$ so that $m'\Lambda'$ is primitive. In this way we normalize $\Gamma(\mathcal F_{m'\Lambda'})=\mathbb Z$, while keeping the total space and central fibre unchanged. The total space $(\mathcal X',\mathcal L')$ of $\mathcal F_{m'\Lambda'}$ is a $\mathbb C^*$-equivariant family over $\mathbb C$ with $\mathcal X_{0}'$ the fibre at $0\in\mathbb C$. By adding a trivial fibre at $\infty$ we can compactify it to a family $\bar{\mathcal X}_i\overset{{\rm pr}'}{\to}\mathbb {CP}^1$. Up to adding to $\mathcal L'$ the pull-back $c{\rm pr}'^*(\infty)$ for sufficiently large $c\in\mathbb N_+$ so that $\bar{\mathcal L}'=\mathcal L'+c{\rm pr}'^*(\infty)$ is ample, the compactified total space $(\bar{\mathcal X}',\bar{\mathcal L}')$ is a polarized $G\times\mathbb C^*$-projective completion of $G/H\times\mathbb C^*$ with moment polytope
$$\Delta(\bar{\mathcal L}'):=\{(\lambda,t)\in\mathfrak X_\mathbb R(B)\oplus\mathbb R|~0\leq t\leq f(\lambda):=C+m'\Lambda'(\lambda),~\lambda\in\Delta(L)\},$$
where $C\in\mathbb N_+$ is a constant. The central fibre $\mathcal X_{0}'$ is the $G\times\mathbb C^*$-invariant divisor that corresponds to the graph of $f$ on $\Delta(L)$. Also, it is obvious that the valuation cone $\mathcal V(G/H\times\mathbb C^*)=\mathcal V(G/H)\times\mathbb Q$. The colours $\mathcal D(G/H\times\mathbb C^*)=\{D\times\mathbb C^*|D\in\mathcal D(G/H)\}$  so that the map $\bar\varrho:\mathcal D(G/H\times\mathbb C^*)\to\mathcal V(G/H)\times\mathbb Q$ is given by
$$\bar\varrho(D\times\mathbb C^*)=(\varrho(D),0),$$
and
\begin{align}\label{colour-type-TC}
D\times\mathbb C^*\in\mathcal D^\star(G/H\times\mathbb C^*)~\text{if and only if}~D\in\mathcal D^\star(G/H).
\end{align}
Denote by $\mathfrak F_X$ the coloured fan of $X$. The coloured fan $\mathfrak F_{\bar{\mathcal X}'}$ consists of the following coloured cones:

\emph{Type-1: Coloured cones that correspond to $G\times\mathbb C^*$-invariant subvarieties in the trivial fibre at $\infty\in\mathbb{CP}^1$.} They are of form
\begin{align}\label{type-1-TC}
({\rm Span}_{\mathbb Q_{\geq0}}(\mathfrak C\times\{0\},(0,1)),\mathfrak R\times\{0\}),
\end{align}
where $(\mathfrak C,\mathfrak R)\in\mathfrak F_X$.

\emph{Type-2: Coloured cones correspond to $\mathbb C^*$ times the $G$-invariant subvarieties in $X$.}
They are of form
\begin{align}\label{type-2-TC}
(\mathfrak C,\mathfrak R)\times\{0\},
\end{align}
when it corresponds to $\mathbb C^*$ times the $G$-invariant subvariety given by $(\mathfrak C,\mathfrak R)$.

\emph{Type-3: Coloured cones that correspond to $G\times\mathbb C^*$-invariant subvarieties in $\mathcal X_{0}'$.}
They are of form
\begin{align}\label{type-3-TC}
(\bar{\mathfrak C},\mathfrak R\times\{0\}),
\end{align}
where $\bar{\mathfrak C}$ is a inner normal cone of $\Delta(\bar{\mathcal L}')$ at a point on the graph of $f$ so that its relative interior ${\rm RelInt}(\bar{\mathfrak C})\cap(\mathcal V\times\mathbb Q)\not=\emptyset$. Any such a cone is generated by some inner normal cone $\sigma\times\{0\}$ of $\Delta(L)$ and $(m'\Lambda',-1)$. It remains to select the colours. Since every lower dimensional cones in $\mathfrak F_{\bar{\mathcal X}'}$ is the face of a maximal dimensional cone, it suffices to select the colours for every $\sigma$ of maximal dimension. That is, $\sigma$ is the inner normal cone of $\Delta(\mathcal L)$ at a vertex $\tilde v=(v,f(v))$ on the graph of $f$, where $v$ is a vertex of $\Delta(L)$. According to \eqref{Cartier-condition}, for such a $\sigma$,
\begin{align}\label{type-3-TC-colour}
\mathfrak R=\{D\in\mathcal D(G/H)|(\varrho(D),0)\in \sigma,~\varrho(D)(v)+\lambda_D=0\},
\end{align}
where $\lambda_D$ is given in \eqref{weil-div}. 

In particular, $\mathcal X_{0}'$ is determined by the colourless cone $(\bar{\mathfrak C}_0,\emptyset):=(\mathbb Q_{\geq0}(m'\Lambda',$ $-1),\emptyset)$ of Type-3. 

\begin{prop}\label{colour-G/H0}
The set $\Pi^p_{G/H_0}=\Pi^p_{G/H}$. Also, for each colour $D_0\in\mathcal D^a(G/H_0)$, there is a unique $D\in\mathcal D^a(G/H)$ so
that $D$ is moved by $P_{\alpha_D}$ for the simple root $\alpha_D(=\alpha_{D_0})\in\Pi_{G/H_0}$, and $\varrho_0(D_0)=\varrho(D)$. In particular, the sets $\Pi^p_{G/H_0}$, $\mathcal D^a(G/H_0)$ and
the map $\varrho_0^a(G/H_0)$ depends only on the face of $\mathcal V_\mathbb R(G/H)$, that contains $\Lambda$ in its relative interior.
\end{prop}

\begin{proof}
We first prove the Proposition for $G$-equivariant special $\mathbb Z$-test configuration $\mathcal F_{m'\Lambda'}$, where $\Lambda'$ is rational and $m'\in\mathbb N_+$ is chosen so that $m'\Lambda'$ is primitive. As discussed above, the total space compactified to a polarized $G\times\mathbb C^*$-spherical variety $\bar{\mathcal X}'$ and $\mathcal X_0$ corresponds to the colourless cone $(\mathbb Q_{\geq0}(m'\Lambda',-1),\emptyset)$ in the coloured fan of $\bar{\mathcal X}'$. The Proposition then follows directly from \cite[Theorem 1.1]{Ga-Ho-datum}:
$$\Pi_{G/H_0}^p=\{\alpha\in Pi_G|\mathcal D(G/H\times\mathbb C^*;\alpha)=\emptyset\}=\Pi_{G/H}^p;$$
Also,for each $D_0\in\mathcal D^a(G/H_0)$ that moves by some $P_\alpha$ with $\alpha\in\Pi_{G/H_0}$, there is a unique $D\in\mathcal D(G/H;\alpha)$ so that $D_0$ is contained in the closure of $D\times\mathbb C^*$ in $\bar{\mathcal X}'$. In particular, there is a bijection
\begin{align}\label{psi-map}
\psi:\mathcal D^a(G/H_0)\to\cup_{\alpha\in\Pi_{G/H_0}}\mathcal D(G/H;\alpha).
\end{align}
We further identify $\mathfrak N(G/H_0)$ with $\mathfrak N(G/H)$, since they are the $\mathbb Z$-dual spaces of isomorphic lattices $\mathfrak M(G/H_0)$ and $\mathfrak M(G/H)$, respectively. Then it is direct to check that $\varrho_0(D_0)=\varrho(D)$.

In particular, we see that the sets $\Pi_{G/H_0}$, $\mathcal D^a(G/H_0)$ and the map $\varrho_0$ depend only on the collection of spherical roots in $\Pi_{G/H}$ that are orthogonal to $\Lambda$. Thus, we have proved the Proposition for $G$-equivariant special $\mathbb Z$-test configurations.

For a general $\mathcal F_\Lambda$ with $\Lambda\in\mathcal V_\mathbb R(G/H)$, by Proposition \ref{F-Lambda-appro}, there is a $\Lambda'\in\mathcal V_\mathbb R(G/H)$ so that $\mathcal F_\Lambda$ and $\mathcal F_{\Lambda'}$ have the same central fibre. With the help of Corollary \ref{Lambda-appro}, the Proposition then reduces to the case of special $\mathbb Z$-test configurations discussed above.
\end{proof}

\subsection{Coloured cone of the central fibre}
In this section, we will prove
\begin{thm}\label{finiteness-X0}
Let $(X,L)$ be a polarized $G$-spherical variety with moment polytope $\Delta(L)$. Let $\mathcal F_{\Lambda_1}$ and $\mathcal F_{\Lambda_2}$ be two $G$-equivariant special $\mathbb R$-test configuration of $(X,L)$. Then $\mathcal F_{\Lambda_1}$ and $\mathcal F_{\Lambda_2}$ have the same central fibre if and only if $\Lambda_1$ and $\Lambda_2$ lie in the relative interior of a same face of $\mathcal V_\mathbb R(G/H)$. In particular, up to $G$-equivariant isomorphism, the set of possible central fibres of $G$-equivariant special $\mathbb R$-test configurations
$$\{\mathcal X_0|~\mathcal X_0~\text{is the central fibre of}~\mathcal F_\Lambda~\text{with}~\Lambda\in\mathcal V_\mathbb R(G/H)\}$$
is a finite set.
\end{thm}

\begin{proof}
Let $\mathfrak F$ be any face of $\mathcal V_\mathbb R(G/H)$. By Proposition \ref{F-Lambda-appro}, it suffices to show that for any two rational elements $\Lambda_i\in{\rm RelInt}(\mathfrak F),~i=1,2$, the corresponding central fibres $\mathcal X_{0,i}$ of $\mathcal F_{\Lambda_i},~i=1,2$, are isomorphic. By Proposition \ref{homo-sph-datum-G/H0}, they are projective completion of a same spherical homogenous space $G/H_0$. In view of Theorem \ref{coloured-fan-classification}, it suffices to show that $\mathcal X_{0,1}$ and $\mathcal X_{0,2}$ have the same coloured fan.

As before, we scale each $\Lambda_i$ by a positive $m_i$ so that $m_i\Lambda_i$ is primitive. Then we normalize $\Gamma(\mathcal F_{m_i\Lambda_i})=\mathbb Z$. Also we compactifify the total space of $\mathcal F_{m_i\Lambda_i}$ to a polarized $G\times\mathbb C^*$-projective spherical completion $(\bar{\mathcal X}_i,\bar{\mathcal L}_i)$ of $G/H\times\mathbb C^*$ with moment polytope
$$\Delta(\bar{\mathcal L}_i):=\{(\lambda,t)\in\mathfrak X_\mathbb R(B)\oplus\mathbb R|~0\leq t\leq f_{i}(\lambda):=C+m_i\Lambda_i(\lambda),~\lambda\in\Delta(L)\},$$
where $C_i\in\mathbb N_+$. The coloured fan $\mathfrak F_{\bar{\mathcal X}_{m_i\Lambda_i}}$ is given according to \eqref{type-1-TC}-\eqref{type-3-TC-colour}. We recall that $\mathcal X_{0,i}$ is determined by the colourless cone $(\bar{\mathfrak C}_0,\emptyset):=(\mathbb Q_{\geq0}(m_i\Lambda_i,$ $-1),\emptyset)$ of Type-3. The coloured cone in $\mathfrak F_{\bar{\mathcal X}_i}$ containing $(\bar{\mathfrak C}_0,\emptyset)$ are precisely those of Type-3.

We can then determine the coloured fan $\mathfrak F_{\mathcal X_{0,i}}$ of $\mathcal X_{0,i}$ by using \cite[Theorem 1.2]{Ga-Ho-datum}. Rewrite
\begin{align}\label{coloured-cone-TCi}
&(\bar{\mathfrak C}_i(\sigma),\bar{\mathfrak R}_i(\sigma))\notag\\:=&({\rm Span}_{\mathbb Q_{\geq0}}(\sigma\times\{0\},(m_i\Lambda_i,-1)),(\mathcal D(G/H)\cap \sigma)\times\{0\})
\end{align}
for any inner normal cone $\sigma$ of $\Delta(L)$. Then all coloured cones of Type-3 as
\begin{align*}
\mathfrak F_i^3:=\{(\bar{\mathfrak C}_i(\sigma),\bar{\mathfrak R}_i(\sigma))|\sigma\in {\mathfrak F^3_X(\Lambda_i)}\},
\end{align*}
where
\begin{align}
\mathfrak F^3_X(\Lambda_i):=\{&\sigma~\text{is an inner normal cone of}~\Delta(L)|\notag\\
&{\rm RelInt}({\rm Span}_{\mathbb Q_{\geq0}}(\sigma\times\{0\},(m_i\Lambda_i,-1)))\cap(\mathcal V(G/H)\times\mathbb Q)\not=\emptyset\}.\label{F-Xi-3}
\end{align}

Denote by $\pi_i:\mathfrak N(G/H\times\mathbb C^*)(\cong\mathfrak N\oplus\mathbb Z)\to \mathfrak N(G/H)(\cong\mathfrak N(G/H_0))$ the projection with $\ker(\pi_i)=\mathbb Z(m_i\Lambda_i,-1)$. By \cite[Theorem 1.2]{Ga-Ho-datum}, $\mathfrak F_{\mathcal X_{0,i}}$ consists of all coloured cones
\begin{align}\label{coloured-cone-in-FX1}
(\pi_i(\bar{\mathfrak C}_i(\sigma)),\mathfrak R_i(\sigma)),~\sigma\in\mathfrak F^3_X(\Lambda_i),
\end{align}
where
\begin{align}\label{pi-cone-in-F_X1}
\pi_i(\bar{\mathfrak C}_i(\sigma))=\sigma\subset\mathfrak N_\mathbb Q(G/H_0)(\cong\mathfrak N_\mathbb Q(G/H)),
\end{align}
and
\begin{align}\label{colour-in-FXi}
\mathfrak R_i(\sigma):=&\psi^{-1}(\bar{\mathfrak R}_i(\sigma))\bigcup_{\alpha\in\Pi_G,\mathcal D(G/H\times\mathbb C^*;\alpha)\subset\bar{\mathfrak R}_i(\sigma)}\mathcal D_0(G/H_0;\alpha)\notag\\
=&\psi^{-1}(\bar{\mathfrak R}_i(\sigma))\bigcup_{\alpha\in\Pi_G,\mathcal D(G/H;\alpha)\subset(\mathcal D(G/H)\cap\sigma)}\mathcal D_0(G/H_0;\alpha).
\end{align}
Here $\psi$ is the bijection \eqref{psi-map} that originally constructed in \cite{Ga-Ho-datum}.

Now we prove that $\mathfrak F_{\mathcal X_{0,1}}=\mathfrak F_{\mathcal X_{0,2}}$. By Lemma \ref{cone-lem} in the Appendix and the relation \eqref{F-Xi-3}, we see that $\mathfrak F^3_X(\Lambda_1)=\mathfrak F^3_X(\Lambda_2)$. By \eqref{coloured-cone-in-FX1}-\eqref{pi-cone-in-F_X1} we see that both fans consist the same collection of underlying cones
$$\{\sigma|\sigma\in\mathfrak F_{\mathcal X_{0,1}}(=\mathfrak F_{\mathcal X_{0,2}})\}.$$

It remains to prove that for each $\sigma\in\mathfrak F_{\mathcal X_{0,1}}(=\mathfrak F_{\mathcal X_{0,2}})$, $\frak R_1(\sigma)=\frak R_2(\sigma)$. By \eqref{coloured-cone-TCi} we have $\bar{\mathfrak R}_1(\sigma)=\bar{\mathfrak R}_2(\sigma)$. One the other hand, $\psi$ does not depends on the choice of $\Lambda\in{\rm RelInt}(\mathfrak F)$. Hence $\psi^{-1}(\bar{\mathfrak R}_1(\sigma))=\psi^{-1}(\bar{\mathfrak R}_2(\sigma))$. It is obvious that the remaining part in \eqref{colour-in-FXi} also keeps invariant when varying $\Lambda_i\in{\rm RelInt}(\mathfrak F)$. Hence we have $(\pi_1(\bar{\mathfrak C}_1(\sigma)),\mathfrak R_1(\sigma))=(\pi_2(\bar{\mathfrak C}_2(\sigma)),\mathfrak R_2(\sigma))$ for every $\sigma$, and $\mathfrak F_{\mathcal X_{0,1}}=\mathfrak F_{\mathcal X_{0,2}}$.
\end{proof}




\section{Applications}

In this section, we give several applications of the classification results derived in Section 3.
\subsection{Base change of a $G$-equivariant $\mathbb Z$-test configuration}
Let $(X,L)$ be a polarized spherical embedding of $G/H$ and $(\mathcal X,\mathcal L)$ be a $G$-equivariant normal $\mathbb Z$-test configuration of it. 
Then there is a concave, rational piecewise linear function $f:\Delta_L\to\mathbb R_{\geq0}$ so that the corresponding filtration as points of discontinuity
\begin{align}\label{ZTC-fil-F}
s_{\lambda}^{(k)}=[kf(\frac1k)],~\lambda\in \Delta_{L,k}.
\end{align}
Clearly the corresponding filtration $\mathcal F$ satisfies $\Gamma(\mathcal F)\subset\mathbb Z$. Up to adding a constant, $f$ is precisely the associated function of $(\mathcal X,\mathcal L)$ introduced in \cite[Theorem 4.1]{Del-2020-09}.

In general the central fibre $\mathcal X_0$ is not-reduced. It is known that after a base change $z\to z^d$ for sufficiently divisible $d\in\mathbb N_+$ on $\mathbb C^*$ and then taking normalization, one gets a normal $\mathbb Z$-test
configuration $(\mathcal X_{(d)},\mathcal L_{(d)})$ with reduced central fibre (cf. \cite[Proof of Proposition 7.16]{Boucksom-Hisamoto-Jonsson}). The base change operation is quite important, for example the Futaki invariant of $(\mathcal X_{(d)},\mathcal L_{(d)})$ equals to the
non-Archimedean Mabuchi functional (cf. \cite[Section 7.3]{Boucksom-Hisamoto-Jonsson}). Using the homogeneity \cite{Del-2020-09} derives the non-Archimedean functional of $(\mathcal X,\mathcal L)$.

As we can not find in the literature a precise expression of the associated function of $(\mathcal X_{(d)},\mathcal L_{(d)})$, we prove in the following
\begin{prop}\label{base-change}
Let $(X,L)$ be a polarized spherical embedding of $G/H$ and $(\mathcal X,\mathcal L)$ be a $G$-equivariant normal $\mathbb Z$-test configuration of $(\mathcal X,\mathcal L)$.
Suppose that the filtration $\mathcal F$ induced by $(\mathcal X,\mathcal L)$ satisfying \eqref{ZTC-fil-F}. Then $(\mathcal X_{(d)},\mathcal L_{(d)})$ is induced by the filtration $\mathcal F_{(d)}$ with
points of discontinuity
\begin{align}\label{ZTC-fil-Fd}
s_{\lambda|(d)}^{(k)}=[kdf(\frac1k)],~\lambda\in \Delta_{L,k}.
\end{align}
\end{prop}

\begin{proof}
The Rees algebra of $\mathcal F$ is
$${\rm R}(\mathcal F)=\oplus_{k\in\mathbb N}\oplus_{s\in\mathbb Z}\oplus_{\lambda\in \Delta_{L,k},[kf(\lambda/k)]\geq s}t^{-s}V_\lambda$$

Taking a base change, ${\rm R}(\mathcal F)$ is transferred into
$${\rm R}'=\oplus_{k\in\mathbb N}\oplus_{s\in\mathbb Z}\oplus_{\lambda\in \Delta_{L,k},[kf(\lambda/k)]\geq s}t^{-ds}V_\lambda.$$
It is the Rees algebra of the filtration $\mathcal F'$ which has point of discontinuity $s'^{(k)}_\lambda=ds^{(k)}_\lambda$ on $V_\lambda$, where $\lambda\in \Delta_{L,k}$.

It suffices to take integral closure of ${\rm R}'$ in the normal ring
$$\hat{\rm R}=\oplus_{k\in\mathbb N}\oplus_{s\in\mathbb Z}\oplus_{\lambda\in \Delta_{L,k}}t^{-s} V_\lambda$$
From \emph{Step-1.1} in the proof of Theorem \ref{G-classify} we know that each
\begin{align*}
t^{-s}\sigma\in t^{-s}V_\lambda,~\lambda\in \Delta_{L,k}~\text{and}~(\mathbb Z\ni)s\leq kdf(\frac1{k}\lambda)
\end{align*}
is integral in ${\rm R}'$. Hence the integral closure of ${\rm R}'$ is
\begin{align*}
\overline{{\rm R}'}=\oplus_{k\in\mathbb N}\oplus_{s\in\mathbb Z}\oplus_{\lambda\in \Delta_{L,k},[kdf(\lambda/k)]\geq s}t^{-s}V_\lambda.
\end{align*}
Thus, the filtration $\mathcal F_{(d)}$ of $(\mathcal X_{(d)},\mathcal L_{(d)})$ is the $\mathbb Z$-test configuration defined by $df$ and $\mathbb Z$. Hence we get \eqref{ZTC-fil-Fd}.

\end{proof}


\subsection{The increasing sequence}
In the following we approximate a  $G$-equivariant normal $\mathbb R$-test configuration by an increasing sequence of $\mathbb Z$-test configurations constructed in \cite[Definition-Proposition 2.15]{Han-Li}. More precisely, given such an $\mathbb R$-test configuration $\mathcal F$, we can construct a sequence of $G$-equivariant normal $\mathbb Z$-test configurations $\{\mathcal F_p\}_{p\in\mathbb N_+}$ so that
the filtration of $\mathcal F_p$ on $\oplus_{k\in\mathbb N}R_{pk}$ is induced by $\mathcal F_\mathbb Z$ on the $R_p$-piece (cf. \cite[Section 2.2]{Han-Li}). First, consider
$$(\mathcal F_\mathbb Z)^sR_{p}=\mathcal F^{\lceil s\rceil}R_{p}.$$
It induces an $\mathbb R$-test configuration $\hat{\mathcal F}_p$ of $(X,L^p)$. Then $\mathcal F_p$ is defined to be normalization of $\hat{\mathcal F}_p$.

Let $\Gamma_p$ be the group generated by $\{s^{(p)}_\lambda|\lambda\in \Delta_{L,p}\}$. Then $\Gamma_p\subset\mathbb Z$.
The Rees algebra of $\hat{\mathcal F}_p$ is
\begin{align*}
{\rm R}_p:={\rm R}(\hat{\mathcal F}_p)=\oplus_{k\in\mathbb N}\oplus_{s\in\Gamma_p}\oplus_{\lambda\in \Delta_{L,kp},s'^{(kp)}_\lambda\geq s}t^{-s}V_\lambda,
\end{align*}
where
$$s'^{(kp)}_\lambda=\max\{\sum_{i=1}^ks^{(p)}_{\mu_i}|\mu_i\in  \Delta_{L,p}, V_\lambda\subset V_{\mu_1}\cdot...\cdot V_{\mu_k}\}.$$

Let $f$ be the functions associated to $\mathcal F$ given by Theorem \ref{G-classify}. Define
\begin{align}\label{fp-func}
f_p(\mu)=\min\{\varphi(\mu)|\varphi(\mu)~\text{is concave and}~\varphi(\frac1p\lambda)\geq\frac{[s^{(p)}_\lambda]}{p},~\lambda\in \Delta_{L,p}\}.
\end{align}
From \emph{Step-1.1} in the proof of Theorem \ref{G-classify} we know that each
$$t^{-s}\sigma\in t^{-s}V_\lambda,~\lambda\in \Delta_{L,kp}~\text{and}~(\Gamma_p\ni)s\leq kpf_p(\frac1{kp}\lambda)$$
is integral in ${\rm R}_p$. Hence the integral closure of ${\rm R}_p$ is
\begin{align*}
\overline{{\rm R}_p}=\oplus_{k\in\mathbb N}\oplus_{s\in\Gamma_p}\oplus_{\lambda\in \Delta_{L,kp},kpf_p(\lambda/kp)\geq s}t^{-s}V_\lambda.
\end{align*}
Thus, $\mathcal F_p$ is the $\mathbb Z$-test configuration defined by $f_p$ and $\Gamma_p$.

Choose a set of nonnegative generators $\{e_j\}_{j=1}^{r_\mathcal F}$ of $\Gamma(\mathcal F)$ and denote $\delta_{\mathcal F}=\max_j\{e_j\}$. Then $$0\leq f(\lambda)-\frac{1+\delta_\mathcal F}{p}\leq \frac{[s^{(p)}_\lambda]}{p}\leq f(\lambda),$$
we have
\begin{align}\label{fp-con-1}
0\leq f(\lambda)-f_p(\lambda)\leq \frac{1+\delta_\mathcal F}{p},~\lambda\in  {\Delta_L}.
\end{align}
Hence we get the uniform convergency,
\begin{align}\label{fp-con}
f_p\rightrightarrows f,~\text{on}~  {\Delta_L}~\text{as}~p\to+\infty.
\end{align}

In summary, we conclude
\begin{prop}\label{RTC-incre-seq}
Let $\mathcal F$ be a $G$-equivariant normal $\mathbb R$-test configuration of $(X,L)$ given by the function $f$ and group $\Gamma(\mathcal F)$. Then the $p$-th term $\mathcal F_p$ of the increasing sequence is given by
$f_p$ in \eqref{fp-func} and $\mathbb Z$. Consequently, the sequence $\{f_p\}_{p=1}^{+\infty}$ converges uniformly to $f$ on $\Delta_L$.
\end{prop}

\subsection{H-invariants of $\mathbb Q$-Fano spherical varieties}
One important application of Theorem \ref{G-classify} is to compute the H-invariant of $\mathbb Q$-Fano spherical variety
and find its semistable degeneration. Let $X$ be a $\mathbb Q$-Fano spherical embedding of $G/H$. Take $\Delta_+=\Delta_{-K_X}$ and $\Delta_{+,k}=\Delta_{-K_X,k}$ for $k\in\mathbb N$. For a
normal $\mathbb R$-test configuration $\mathcal F$, we associate to it two functionals ${\rm L^{NA}}(\cdot)$ and ${\rm S^{NA}}(\cdot)$.
Denote by $M_{\mathbb Q}^{\rm div}$ the set of $\mathbb Q$-divisorial valuations of $M$ and $A_M(\cdot)$ the log-discrepancy. Also let $\phi_\mathcal F,\phi_0$ be the
non-Archimedean metric of $\mathcal F$ and the trivial test configuration (cf. \cite[Definition 2.17]{Han-Li}), respectively. Then
\begin{align}\label{L-inv-NA}
{\rm L^{NA}}(\mathcal F):=\inf_{ \mathfrak v\in M_{\mathbb Q}^{\rm div}}(A_M( \mathfrak v)+(\phi_\mathcal F-\phi_0)( \mathfrak v)),
\end{align}
Also, denote by $\Delta^{\rm O}(\mathcal F^{(t)})$ the Okounkov body of the linear series (cf. \cite[Section 2.4]{Han-Li}),
\begin{align}\label{okounkov-body-layer}
\mathcal F^{(t)}:=\{\mathcal F^{tk}R_k\}_{k\in\mathbb N_+}.
\end{align}
By Definition \ref{filtrantion-def} (3), we see that when $t\ll0$, $\Delta^{\rm O}(\mathcal F^{(t)})$ is just the Okounkov body $\Delta^{\rm O}$ of $(M,K_M^{-1})$ introduced in \cite{Okounkov},  and $\Delta^{\rm O}(\mathcal F^{(t)})=\{O\}$ when $t\gg1$. Define a function $G_\mathcal F:\Delta^{\rm O}\to\mathbb R$ by
\begin{align}\label{okounkov-body-func}
G_\mathcal F(z):=\sup\{t|z\in\Delta^{\rm O}(\mathcal F^{(t)})\},~z\in\Delta^{\rm O}
\end{align}
and set
\begin{align}\label{S-inv-NA}
{\rm S^{NA}}(\mathcal F):=-\ln\left(\frac{n!}{V}\int_{\Delta^{\rm O}}e^{-G_\mathcal F(z)}dz\right).
\end{align}
The H-invariant is defined as
\begin{align}\label{H-inv-NA}
{\rm H^{NA}}(\mathcal F):={\rm L^{NA}}(\mathcal F)-{\rm S^{\rm NA}}(\mathcal F),
\end{align}

It is proved by \cite{Blum-Liu-Xu-Zhuang,Han-Li} that ${\rm H^{NA}}(\cdot)$ admits a unique minimizer $\mathcal F_{\min}$ among filtrations. Moreover $\mathcal F_{\min}$ is a special $\mathbb R$-test configuration.
The minimizer $\mathcal F_{\min}$ is the semistable degeneration of $X$ in the sense that its central fibre $\mathcal X_0$ is modified K-semistable with respect to the induced vector field $\Lambda_0$. Indeed $(\mathcal X_0,\Lambda_0)$ gives the first-step
limit in the study of K\"ahler-Ricci flow (cf. \cite{Blum-Liu-Xu-Zhuang, Chen-Wang-Sun, Han-Li}).  By \cite[Proposition 2.9]{LL-arXiv-2021}, on a spherical variety $X$ the minimizer $\mathcal F_{\min}$ is $G$-equivariant. Thus we have

\begin{thm}\label{H-inv-F-thm}
Let $X$ be a $\mathbb Q$-Fano spherical variety with $\Delta_+$ its moment polytope. Then for the $G$-equivariant normal $\mathbb R$-test configuration $\mathcal F$
associated to $f$ and $\Gamma(\mathcal F)$,
\begin{align}\label{H-reduction-eq}
{\rm H^{NA}}(\mathcal F)\geq\ln\left(\frac1V\int_{\Delta_+}e^{-f(y)+f(2\kappa_P)}\pi(y)dy\right),
\end{align}
where $\kappa_P$ is defined in \eqref{kappa} and $$\pi(y)=\prod_{\alpha\in\Phi^G_+,\alpha\not\perp\Delta_+}\langle \alpha,y\rangle.$$
Moreover, the equality holds if $\mathcal F$ is defined by a valuation on $G/H$. That is, $\mathcal F=\mathcal F_{\Lambda_0}$ for some $\Lambda_0\in\mathcal V_\mathbb R(G/H)$.
\end{thm}
\begin{proof}[Sketch of proof]
Since most of the proof goes as in \cite[Section 5]{LL-arXiv-2021}, we just sketch it in the following. For ${\rm L^{NA}}(\cdot)$, choose the approximating sequence of $\mathbb Z$-test configurations $\{\mathcal F_p\}_{p\in\mathbb N_+}$
of $\mathcal F$ defined in Section 6.2, by \cite[Remark 2.29]{Han-Li},
\begin{align*}
{\rm L^{NA}}(\mathcal F)\geq\lim_{p\to+\infty}{\rm L^{NA}}(\mathcal F_p),
\end{align*}
and the equality holds if $\mathcal F$ is defined by a valuation on $G/H$. To compute ${\rm L^{NA}}(\mathcal F_p)$, we use the lct formula as in \cite[Section 4.3]{Yao}.
Using the canonical $B$-semiinvariant section $\mathfrak d_0$ of $-K_X$ with $B$-weight $\kappa_P$, as in \cite[Lemma 5.3]{LL-arXiv-2021} we gets
\begin{align}\label{L-Fp}
{\rm L^{NA}}(\mathcal F_p)=f_p(\kappa_P).
\end{align}
Combining with \eqref{fp-con},
\begin{align}\label{L-reduction-ineq}
{\rm L^{NA}}(\mathcal F)\geq f(\kappa_P),
\end{align}
and the equality holds if $\mathcal F$ is defined by a valuation on $G/H$.

For ${\rm S^{NA}}(\cdot)$, we need to find the Okounkov bodies of $\mathcal F^{(t)}=\{\mathcal F^{kt}R_k\}_{k\in\mathbb N}$. By definition
$$\mathcal F^{kt}R_k=\oplus_{\lambda\in \Delta_{+,k},s^{(k)}_\lambda\geq kt}V_\lambda.$$
Let $\{e_1,...,e_{r_\mathcal F}\}$ be a set of $\mathbb Z$-basis of $\Gamma(\mathcal F)$ and $\delta_0:=\max_{i=1,...,r_\mathcal F}|e_i|$.
Recall \eqref{s-k-G-final},
$$kf(\frac\lambda k)-\delta_0\leq s^{(k)}_\lambda\leq kf(\frac\lambda k),~\lambda\in \Delta_{+,k}.$$
Thus the Okounkov bodies
\begin{align}
&{\rm Conv}(\cup_{\lambda\in \Delta_{+,k};f(\lambda/k)\geq t-\frac{\delta_0}k}(\lambda,\Delta^{\rm GT}(\lambda)))\notag\\
\subset&
\Delta^{\rm O}(\mathcal F^{kt}R_{k})
\subset{\rm Conv}\left(\cup_{\lambda\in \Delta_{+,k};f(\lambda/k)\geq t}(\lambda,\Delta^{\rm GT}(\lambda))\right).
\end{align}
We get
\begin{align}\label{okounkov-body-serise}
\Delta^{\rm O}(\mathcal F^{(t)})=\overline{{\rm Conv}\left(\cup_{k=1}^{+\infty}\frac1k\Delta^{\rm O}(\mathcal F^{kt}R_{k})\right)}=\Delta^{\rm O}\cap\{f(\lambda)\geq t\}=:\Delta^{\rm O}_{f\geq t},
\end{align}
and consequently
\begin{align}\label{G-F-Lambda}
G_{\mathcal F}(z)=\sup\{t|z\in\Delta^{\rm O}_{f\geq t}\}=f(\lambda),~\text{for}~z=(\lambda,z')\in\{\lambda\}\times\Delta^{\rm GT}(\lambda)\subset\Delta^{\rm O}.
\end{align}

On the other hand, by \cite[Theorem 2.5]{Han-Li},
\begin{align}\label{dirac-measure-converge}
dz=\lim_{k\to+\infty}\nu_k,
\end{align}
where $\mu_k$ is the Dirac type measure
\begin{align}\label{dirac-measure}
\nu_k:=\frac{1}{k^n}\sum_{z\in k\Delta^{\rm O}~\text{is an integral point}~}\delta_{\frac zk}.
\end{align}
Also, recall the Weyl character formula \cite[Section 3.4.4]{Zhelobenko-Shtern},
\begin{align}\label{dim-V-lambda}
\dim(V_\lambda)=\frac{\prod_{\alpha\in\Phi^G_+,\alpha\not\perp\Delta_+}\langle\alpha,\rho+k\lambda\rangle}{\prod_{\alpha\in\Phi^G_+,\alpha\not\perp\Delta_+}\langle\alpha,\rho\rangle},~\forall~\text{dominant $G$-weight}~\lambda,
\end{align}
where $\rho=\frac12\sum_{\alpha\in\Phi_+^G}\alpha$.

Plugging \eqref{G-F-Lambda}, \eqref{dirac-measure-converge}-\eqref{dim-V-lambda} and \eqref{okounkov-body-fibre} into \eqref{S-inv-NA} and using the fibration structure of $\Delta^{\rm O}$ over $\Delta_+$, one gets (cf. \cite[Section 1.4]{Pukhlikov-Khovanskii} and \cite[Lemma 5.2]{LL-arXiv-2021})
\begin{align}\label{S-inv-F}
{\rm S^{NA}}(\mathcal F)=-\ln\frac1V\int_{\Delta_+}e^{-f(y)}\pi\,dy+\ln{\prod_{\alpha\in\Phi^G_+,\alpha\not\perp\Delta_+}\langle\alpha,\rho\rangle}.
\end{align}

Theorem \ref{H-inv-F-thm} then follows from \eqref{L-reduction-ineq} and \eqref{S-inv-F}.
\end{proof}

From Theorem \ref{H-inv-F-thm} and concavity of $f$ one directly sees
$${\rm H^{NA}}(\mathcal F)\geq\int_{\Delta_+}e^{-\Lambda(y-\kappa_P)}\pi(y) dy={\rm H^{NA}}(\mathcal F_\Lambda)$$
whenever $\Lambda$ is the slope of a supporting function of $f$ at $\kappa_P$. To find the minimizer we further minimize ${\rm H^{NA}}(\mathcal F_\Lambda)$
for $\Lambda\in\mathcal V_\mathbb R(G/H)$. Since ${\rm H^{NA}}(\mathcal F_\Lambda)$ is both convex and proper with respect to $\Lambda$ it admits a unique minimizer $\Lambda_0$ and $\mathcal F_{\Lambda_0}$ is the semistable degeneration of $X$.

As in \cite[Lemma 5.4]{LL-arXiv-2021}, one can prove
\begin{align}\label{bar-Lambda0}
\mathbf{b}(\Lambda_0):=\frac{\int_{\Delta_+}y_ie^{-\Lambda_0(y)}\pi dy}{\int_{\Delta_+}e^{-\Lambda_0(y)}\pi dy}\in\kappa_P+\overline{(-\mathcal V_\mathbb R(G/H_0)^\vee)}.
\end{align}

Recall the combinatory data of $\mathcal X_0$ studied in Section 5, implies $(\mathcal X_0,\Lambda_0)$ is modified K-semistable. As in \cite[Proposition 5.5]{LL-arXiv-2021} we have
\begin{prop}\label{limit-KRF}
If \eqref{bar-Lambda0} is strict, i.e.
$$\mathbf{b}(\Lambda_0)\in\kappa_P+{\rm RelInt}({-\mathcal V(G/H_0)^\vee}),$$
then $\mathcal X_0$ is the limiting space of the normalized K\"ahler-Ricci flow on $X$. Moreover, $\mathcal X_0$ admits a K\"ahler-Ricci soliton with soliton vector field $-\Lambda_0$.
\end{prop}

For more details on semistable degenerations and the limiting problem of K\"ahler- Ricci flow, we refer to the readers \cite{Li-Wang}.

\begin{rem}
It is showed in \cite{Blum-Liu-Xu-Zhuang, Han-Li} that the minima of ${\rm H^{NA}}(\cdot)$ must be achieved by an $\mathbb R$-test configuration induced by a quasimonomial  $\mathbb R$-valuation on $X$. Combining with \cite[Corollary 2.10]{LL-arXiv-2021} and Remark \ref{R-val}, it suffices to minimize ${\rm H}^{\rm NA}(\cdot)$ among $\{\mathcal F_\Lambda|\Lambda\in\mathcal V_\mathbb R\}$. Therefore we get an alternative proof of Theorem \ref{main-thm-4} without using the full classification result Theorem \ref{G-classify}.
\end{rem}

Finally, we have
\begin{proof}[Proof of Theorem \ref{main-thm-6}]
Recall that the limiting space $X_\infty$ can be derived from $X$ via two-steps degeneration (cf. \cite{Blum-Liu-Xu-Zhuang, Chen-Wang-Sun,Han-Li}). More precisely, the semistable degeneration $\mathcal F_{\Lambda_0}$ of $(X,K^{-1}_X)$ in Theorem \ref{H-inv-F-thm}, that degenerates $X$ to its semistable limit $\mathcal X_0$. We see that $\mathcal X_0$ is a $\mathbb Q$-Fano $G$-spherical embedding of some spherical homogenous space $G/H_0$. Then there is a $G$-equivariant special $\mathbb R$-test configuration $\mathcal F'=\mathcal F'_{\Lambda'_0}$ of $(\mathcal X_0,K^{-1}_{\mathcal X_0})$ with $\Lambda'\in\mathcal V_\mathbb R(G/H_0)$ that degenerates $\mathcal X_0$ to $X_\infty$. Again, $X_\infty$ is a $\mathbb Q$-Fano $G$-spherical embedding of some spherical homogenous space $G/H'_0$, and $X_\infty$ admits a K\"ahler-Ricci soliton. By Proposition \ref{val-cone-G/H0-prop},
$$\Pi_{G/H'_0}\subset\Pi_{G/H_0}\subset\Pi_{G/H}.$$
If $X_\infty$ is also a spherical embedding of $G/H$, then the above inclusions are both equalities. In particular, by Proposition \ref{val-cone-G/H0-prop} we see that for the semistable degeneration $\mathcal F_{\Lambda_0}$, $\Lambda_0$ is orthogonal to $\Pi_{G/H}$. Then by Theorem \ref{finiteness-X0}, $\mathcal X_0$ is isomorphic to $X$. Applying the above arguments to $\mathcal F'_{\Lambda'_0}$, we see that $\Lambda'_0$ is also orthogonal to $\Pi_{G/H}$, and $X_\infty$ is isomorphic to $X$. A contradiction to our assumption that $X$ does not admit any K\"ahler-Ricci soliton.

\end{proof}



\section{Appendix: A Lemma on convex cones}
\begin{lem}\label{cone-lem}
Let $C_1,C_2$ be two finitely generated (closed) cone in $V=\mathbb R^n$. For each $\Lambda\in C_2$, denote by $C_1(\Lambda)$ the cone generated by $C_1$ and $\Lambda$. Suppose that $\Lambda$ lies in the relative interior of a face $F$ of $C_2$. If the relative interior ${\rm RelInt}(C_1(\Lambda))$ of $C_1(\Lambda)$ does not intersects $C_2$. Then for any $\Lambda'\in\bar F$,  ${\rm RelInt}(C_1(\Lambda'))\cap C_2=\emptyset$.
\end{lem}

\begin{proof}
We first show that there is a $u\in V^*$ so that the hyperplane
$\Pi_u:=\langle u\rangle^\perp$ separates $C_1(\Lambda)$ and $C_2$ so that
$$C_1(\Lambda)\subset\Pi_u^{\geq0}:=\{v\in V|u(v)\geq0\},~C_2\subset\Pi_u^{\leq0}:=\{v\in V|u(v)\leq0\}.$$
It suffice to show that the intersection of dual cones $C_1(\Lambda)^\vee\cap (-C_2^\vee)\not=\{O\}$. Otherwise, taking dual on both sides we see that $C_1(\Lambda)$ and $-C_2$ generates $V$. Choose a $v'\in{\rm RelInt}(C_1(\Lambda))$, there are $v_1\in C_1(\Lambda)$ and $v_2\in C_2$ so that $-v'=v_1-v_2$. Then $v_2=v_1+v'\in{\rm RelInt}(C_1(\Lambda))$, a contradiction. Hence there is a $(0\not=)u\in C_1(\Lambda)^\vee\cap (-C_2^\vee)$ that satisfies our requirements. Clearly, $\Lambda\in\Pi_u$, and consequently $F\subset\Pi_u$.

The proof then depends on $\dim C_1(\Lambda)$:

\emph{Case-1.} $\dim C_1(\Lambda)=n$. In this case $C_1$ intersects $\Pi_u^+:=\{v\in V|u(v)>0\}$. Otherwise, $C_1(\Lambda)\subset\Pi_u$, which contradicts with $\dim C_1(\Lambda)=n$. Hence for any $\Lambda'\in \bar F\subset\Pi_u$, ${\rm RelInt}(C_1(\Lambda'))\subset\Pi_u^+$, which does not intersects $C_2$.

\emph{Case-2.} $\dim C_1(\Lambda)=k<n$. There are two subcases:

\emph{Csee-2.1.} $C_1$ intersects $\Pi_u^+:=\{v\in V|u(v)>0\}$, then for any $\Lambda'\in \bar F\subset\Pi_u$, ${\rm RelInt}(C_1(\Lambda'))\subset\Pi_u^+$ and hence does not intersects $C_2$.

\emph{Csee-2.2.} $C_1\subset\Pi_u$. Then for any $\Lambda'\in\bar F\subset\Pi_u$, $C_1(\Lambda')\subset\Pi_u$. We consider the cone $C_2\cap\Pi_u$ in stead of $C_2$. Then $C_1(\Lambda)$ and $C_2\cap\Pi_u$ are two finitely generated cone in $\Pi_u\cong\mathbb R^{n-1}$. Consider a hyperplane $\Pi_{u'}$ of $\Pi_u$ that separates $C_1(\Lambda)$ and $C_2\cap\Pi_u$. Again we have $\bar F\subset\Pi_{u'}$. If $C_1\cap\Pi_{u'}^+\not=\emptyset$, then we go back to \emph{Case-2.1}. Otherwise we repeat the above progress. The progress continues for at most $(n-k+1)$ times, since at the $(n-k+1)$-th time we will get a $k$-dimensional plane $\Pi$ that contains $C_1(\Lambda)$, $\bar F$ and $C_2\cap\Pi$, which reduces to \emph{Case-1}.

\end{proof}

\subsection*{Author Declaration} The authors declare no conflict of interest. Also, the manuscript has no associated data.

\vskip30pt

\end{document}